\documentclass[reqno, 12pt]{amsart}

\usepackage{amsfonts, amsthm, amsmath, amssymb}

\usepackage{hyperref}
\usepackage{a4wide}

\RequirePackage{mathrsfs} \let\mathcal\mathscr

\DeclareRobustCommand{\SkipTocEntry}[5]{} 

\numberwithin{equation}{section}

\newtheorem{theorem}{Theorem}[section]
\newtheorem{lemma}[theorem]{Lemma}
\newtheorem{proposition}[theorem]{Proposition}

 \newcommand{\set}[1]{\left\{#1\right\}}
\newcommand{\bigset}[1]{\bigl\{ #1 \bigr\}}
\newcommand{\Bigset}[1]{\Bigl\{ #1 \Bigr\}}
\newcommand{\biggset}[1]{\biggl\{ #1 \biggr\}}

\newcommand{\abs}[1]{\left| #1\right|}
\newcommand{\bigabs}[1]{\bigl| #1 \bigr|}
\newcommand{\Bigabs}[1]{\Bigl| #1 \Bigr|}
\newcommand{\biggabs}[1]{\biggl| #1 \biggr|}
\newcommand{\Biggabs}[1]{\Biggl| #1 \Biggr|}

\newcommand{\ceil}[1]{\left\lceil #1 \right\rceil}

\newcommand{\floor}[1]{\left\lfloor #1 \right\rfloor}

\newcommand{\bigbrac}[1]{\bigl( #1 \bigr)}
\newcommand{\Bigbrac}[1]{\Bigl( #1 \Bigr)}
\newcommand{\biggbrac}[1]{\biggl( #1 \biggr)}

\newcommand{\norm}[1]{\left\| #1\right\|}

\newcommand{\rd}{\,\mathrm{d}}
\newtheorem*{theorema}{Theorem A}
\newtheorem*{theoremb}{Theorem B}
\newtheorem*{theoremc}{Theorem C}
\newtheorem*{theoremd}{Theorem D}
\newtheorem*{theoreme}{Theorem E}
\newtheorem*{theoremf}{Theorem F}

\newcommand{\threesum}[3]{\sum_{\substack{#1\\ #2\\ #3}}}

\newcommand{\Z}{\mathbb{Z}}

\newcommand{\C}{\mathbb{C}}
\newcommand{\T}{\mathbb{T}}

\newcommand{\E}{\mathbb{E}}
\renewcommand{\P}{\mathcal P}
\newcommand{\R}{\mathcal R}
\newcommand{\major}{\mathfrak M}
\newcommand{\minor}{\mathfrak m}
\newcommand{\majors}{\mathfrak M^*}
\newcommand{\minors}{\mathfrak m^*}
\newcommand{\eps}{\varepsilon}
\newcommand{\A}{\mathcal A}

\renewcommand{\bar}{\overline}
\renewcommand{\hat}{\widehat}

\title[Sparse Furstenberg-S\'ark\"ozy Configurations]
{A quantitative bound on Furstenberg-S\'ark\"ozy patterns with shifted prime power common differences in primes}

\author{Mengdi Wang}
\address{Institute of Mathematics, \'Ecole Polytechnique F\'ed\'erale de Lausanne (EPFL), CH-1015 Lausanne, Switzerland}
\email{mengdi.wang@epfl.ch}

\subjclass[2000]{Primary 11B05; Secondary 11B30 11L07 11J55}

\begin{document}

\begin{abstract}
Let $k\geq1$ be a fixed integer and $\mathcal P_N$ be the set of primes no more than $N$. We prove that if a set  $\mathcal A\subset\mathcal P_N$ contains no patterns $p_1,p_1+(p_2-1)^k$, where $p_1,p_2$ are prime numbers, then
\[
\frac{|\mathcal A|}{|\mathcal P_N|}\ll\bigbrac{\log\log N}^{-\frac{1}{4k^3+23k^2}}.
\]
\end{abstract}

\maketitle


\section{Introduction}

Lov\'asz conjectured that any integer set without non-zero perfect square differences must have asymptotic density zero. This conjecture was confirmed by  Furstenberg \cite{Fur} and S\'ark\"ozy \cite{S781} independently. Furstenberg used ergodic theory and his result is  a purely qualitative one. Instead,  using the Fourier-analytic density increment method, S\'ark\"ozy proved the following quantitative strengthening.
\begin{theorema}[S\'ark\"ozy \cite{S781}]
If $\A$ is a subset of $[N]=\set{1,\dots,N}$ and lacks non-trivial patterns $x,x+y^2(y\neq0)$, then
\[
\frac{|\A|}{N}\ll\Bigbrac{\frac{(\log\log N)^2}{\log N}}^{1/3}.
\]
\end{theorema}
Later Prendiville \cite{Pre} generalized S\'ark\"ozy's argument to homogeneous polynomials, and he shows that sets without homogeneous polynomial Szemer\'edi configurations have double logarithmic decay.
Pintz, Steiger and Szemer\'edi \cite{PSS} created a double iteration strategy which refined S\'ark\"ozy's proof and improved the above upper bound to
\[
\frac{|\A|}{N}\ll(\log N)^{-\frac{1}{12}\log\log\log\log N}.
\]
Subsequently, this method has been generalized to deal with sets without differences of the form $n^k(k\geq3)$ by Balog, Pelik\'an, Pintz and Szemer\'edi \cite{BPPS}. They proved that for any integer $k\geq3$,  if $\A\subset[N]$ lacks non-trivial patterns $x,x+y^k(y\neq0)$, then
\[
\frac{|\A|}{N}\ll(\log N)^{-\frac{1}{4}\log\log\log\log N}.
\]

Recently, Bloom and Maynard gained a much more efficient density increment argument by showing that no set $\A$ can have many large Fourier coefficients which are rationals with small and distinct denominators. Taking  advantage of this density increment argument they  proved
\begin{theoremb}[Bloom-Maynard \cite{BM}]
If $\A\subset[N]$ lacks non-trivial patterns $x,x+y^2$ with $y\neq0$, then
\[
\frac{|\A|}{N}\ll(\log N)^{-c\log\log\log N}	
\]
for some absolute constant $c>0$.
\end{theoremb}

In the same series of papers (to study difference sets of subsets of integers), S\'ark\"ozy also answered a question that was posed by Erd\H{o}s, proving that if a subset of integers does not contain two elements that differ by one less than a prime number, then this subset has asymptotic density zero. Assume that $\A$ is a subset of $[N]$, in what follows, we shall adopt $\A-\A$ to denote the difference set $\set{a_1-a_2:a_1,a_2\in \A}$. More formally, S\'ark\"ozy proved that
\begin{theoremc}[S\'ark\"ozy \cite{S783}]
If $\A\subset[N]$ and $p-1\not\in \A-\A$ for all primes $p$, then
\[
\frac{|\A|}{N}\ll\frac{(\log\log\log N)^3\log\log\log\log N}{(\log\log N)^2}.
\]
	
\end{theoremc}
By exploiting a dichotomy depending on whether the exceptional zero of Dirichlet $L$-functions occurs or not, Ruzsa and Sanders \cite{RS} proved that the above upper density bound can be of the magnitude $\frac{|\A|}{N}\ll\exp\bigbrac{-O\bigbrac{(\log N)^{1/4}}}$. Very recently, Green \cite{Gr22} improved this bound to $N^{-c}$ by studying the van der Corput property for the shifted primes, and such a bound is indeed out of the strength of the density increment argument. Subsequently, Thorner and Zaman \cite{TZ} show that this constant can be $c=3.4/10^{18}$.

Li and Pan combined the above two questions posed by Lov\'asz and Erd\H{o}s respectively and showed that
\begin{theoremd}[Li-Pan \cite{LP}]
Let $h\in\Z[x]$ be a polynomial with a positive leading term and zero constant term. If $\A\subset[N]$ and $h(p-1)\not\in \A-\A$ for all primes $p$, then
\[
\frac{|\A|}{N}\ll\frac{1}{\log\log\log N}.
\]	
\end{theoremd}
Rice improved the above triple logarithmic decay to a single logarithmic bound.  Rice also generalized this kind of consideration to a larger class of polynomials named $\P$-intersective polynomials (see \cite[Definition 1]{Rice}).

\begin{theoreme}[Rice \cite{Rice}]
Let $h\in\Z[x]$ be a 	$\P$-intersective polynomial of degree $k\geq2$ with positive leading term. If $\A\subset[N]$ and $h(p-1)\not\in \A-\A$ for all primes $p$ with $h(p-1)>0$, then
\[
\frac{|\A|}{N}\ll(\log N)^{-c}
\]
for any $0<c<\frac{1}{2k-2}$.
\end{theoreme}

Similar to Green-Tao theorem \cite{GT08}, which shows that the primes contain arbitrarily long arithmetic progressions, one interesting question is to study whether sequences of prime numbers contain the above-described patterns. As the readers may find from Theorem B, primes trivially contain $x,x+y^2$ with $y\neq0$. By adopting the transference principle \cite{GT06}, Li-Pan and Rice can also prove the corresponding $\P$-intersective polynomial difference results of subsets of prime numbers.
\begin{theoremf}[Li-Pan \cite{LP}, Rice \cite{Rice}]
Let $h\in\Z[x]$ be a zero constant term polynomial. If $\A\subset\P_N$ is a subset of primes and $\A$ does not contain  patterns $p_1,p_1+h(p_2-1)$ for all primes $p_1,p_2$, then
\[
|\A|=o(|\P_N|).
\]	
\end{theoremf}

In this paper, we are planning to provide Theorem F a quantitative relative density bound when $h\in\Z[x]$ takes the form of the perfect $k$-th power polynomial, i.e. $h(x)=x^k$. 
\begin{theorem}\label{main}
Let $k\geq1$ be a fixed integer, and $\P_N$ be the set of primes which are no more than $N$. If $\A\subset\P_N$ and $(p-1)^k\not\in \A-\A$ for all primes $p$, then
\[
\frac{|\A|}{|\mathcal P_N|}\ll\bigbrac{\log\log N}^{-\frac{1}{4k^3+23k^2}}.
\]
\end{theorem}

It would be worth mentioning that the exponent $\frac{1}{4k^3+23k^2}$ can be improved a bit slightly by more careful calculations, however, it seems through our method it can not be better than $\frac{1}{4k^3}$. Besides, we believe that our argument can be generalized to deal with all polynomials with zero constant terms, and hopefully all $\P$-intersective polynomials. This would, however, involve more complicated exponential sum estimates.

The proof of Theorem \ref{main} turns our attention to the restriction estimates of primes and prime powers in arithmetic progressions. Suppose that $1\leq b\leq d$ are coprime integers, $k$ is an integer, and $M=\floor{N^{1/k}}$. Suppose that $P=b+d\cdot[X]\subseteq[N]$ is an arithmetic progression, define
\begin{align}\label{laq}
\Lambda_{b,d}(x)=\begin{cases}
\frac{\phi(d)}{d}\Lambda(b+dx)\quad &\text{ if }	1\leq x\leq X;\\
0&\text{ otherwise,}
\end{cases}
\end{align}
and for any $\alpha\in\T$ define
\begin{align}\label{exp}
	S_d(\alpha)=\sum_{y\in[M]}\frac{\phi(d)}{d}ky^{k-1}\Lambda(dy+1)e(y^k\alpha).
\end{align} 
Then:
\begin{theorem}\label{thm2}
Suppose that $1\leq b\leq d\leq \exp\bigbrac{c\sqrt{\log N}}$ are coprime integers for some $0<c<1$,
\begin{enumerate}
\item If $\nu:\Z\to\C$ is a function	supported on $[X]$ and satisfying $|\nu|\leq\Lambda_{b,d}$ pointwise,  for any real $p>2$ we have
 \[
 \int_\T\Bigabs{\sum_{x\in[X]}\nu(x)e(x\alpha)}^p\,\rd \alpha\ll X^{p-1};
 \]
 \item Let $p>k(k+1)+2$ be a real number. Let $S_d(\alpha)$ be the exponential sum defined in (\ref{exp}),  we have
\[
\int_\T\bigabs{S_d(\alpha)}^{p}\rd\alpha\ll_{p}N^{p-1}.
\]	
\end{enumerate}

\end{theorem}
As a final remark of the introduction, compared to the literature \cite{GT06}, \cite{LP}, \cite{Chow} and so on, we generalize the restriction estimates to modulus $d\ll\Bigbrac{\exp\bigbrac{c'\sqrt{\log N}}}$, and previously, they were only valid when $d\ll(\log N)^A$. 
It is essential to note that \cite[Lemma 5.1]{Chow} is sufficient to prove Theorem \ref{main}, with a slight modification of the function from \cite[(2.5)]{Chow}.  To establish Theorem \ref{thm2}, the primary strategy involves expanding the major-arc analysis introduced in \cite{RS} to a higher degree. In the context of minor arcs, we simplify the estimation process to address Type I and Type II sums. Notably, when dealing with these cases, the arithmetic progression conditions are controllable, particularly when the common difference does not exceed a significantly small power of $N$.

\subsection*{Outline of the argument}

We shall give an outline of our method for proving Theorem \ref{main} in the rest of this section. As mentioned above, the traditional approach to finding (lower complexity) arithmetic configurations in primes is the so-called (Fourier-analysis) transference principle which was essentially developed by Green and Tao \cite{Gr05a}, \cite{GT06}. Roughly speaking, if we can show that any dense subset of integers would always have the targeted configurations (this is always done via density increment argument), one may take such a conclusion as a black box and hope to construct a pseudorandom majorant of subsets of primes to make this majorant indistinguishable from subsets of integers in certain statistical senses --- so that transfer the conclusion to subsets of primes. In this paper,  instead,  we'd like to do the density increment procedure in subsets of primes directly. To our knowledge, this is the first attempt to run the density increment argument in a sparse set. 
 
 Compared to the situation in integers,  to make the density increment argument work successfully in primes many obstructions need to be overcome. Firstly, the ``translation-invariant" property is broken. Assume that our goal is to find $x,x+y^k\in \A$ in the interval $[N]$ and there is an arithmetic progression $P=a+q\cdot[X]$ inside $[N]$ such that $\A$ has an increased density in $P$, then more-or-less we can transfer the matter to find $x,x+q^{k-1}y^k\in \A'$ in the interval $[X]$, where $\A'=(\A-a)/q$. Even though the targeted configuration changes a bit, we are always finding configurations in intervals and such a procedure can be continued as long as the lengths of the intervals are not too short.   In contrast, in the setting of primes, as we need to keep the features $p_1,p_2,p_1+(p_2-1)^k$ being primes throughout, the translation argument is ineffective and we need to find $p_1,p_1+(p_2-1)^k$ in shorter and shorter arithmetic progressions. Because of the limited comprehension of primes in short arithmetic progressions, especially regarding the current knowledge of the modulus, this leads to the observed double logarithmic decay in the main theorem.

We organize this paper as follows. In Section 2, we set notations and exhibit several conclusions of primes in (short) arithmetic progressions. In Section 3, we use the Hardy-Littlewood method to study the Fourier transform of the majorant of shifted prime powers. We then use Bourgain's method \cite{Bou89} and the above Fourier transform information to prove the second result of Theorem \ref{thm2} and give the proof of the first result in Appendix B, as the proof strategies are similar. In Section 5, we'll prove our main result, a local inverse theorem, and then in Section 6 use this inverse result to carry out the density increment argument and so complete the proof of Theorem \ref{main}.

\subsection*{Acknowledgements}
The author would like to thank Kaisa Matom\"aki,  Xuancheng Shao and Lilu Zhao for their helpful discussions. The author also thanks the referee for numerous helpful comments and corrections that significantly improved the paper's expression, as well as the suggestion to add Appendix D.

\section{Notations and Preliminaries}

\subsection{Notations}

For a real number $X\geq1$, we use $[X]$ to denote the discrete interval $\set{1,2,\dots,\floor{X}}$.  $\P$ denotes the set of prime numbers and $\P_N=\P\cap[N]$ denotes the set of primes no more than $N$. For $\alpha\in\T=\mathbb R/\Z$, we write $\norm{\alpha}$ for the distance of $\alpha$ to the nearest integer. For any set of integers $\A$, use $1_\A$ to denote its indicator function.

Following standard arithmetical function conventions, we'll use $\phi$ to denote the Euler totient function; $\Lambda$ to denote the von Mangoldt function; $\mu$ to denote the M\"obius function.

We'll use the counting measure on $\Z$, so for a function $f:\Z\to\C$ its $L^p$-norm is defined to be
\[
\norm{f}_p^p=\sum_x|f(x)|^p,
\]
and $L^\infty$-norm is defined to be
\[
\norm{f}_\infty=\sup_x|f(x)|.
\]
Besides, we'll use Haar probability measure on $\T$, so for a function $F:\T\to\C$ define its $L^p$-norm as
\[
\norm{F}_p^p=\int_\T|F(\alpha)|^p\rd\alpha.
\]

If $f:A\rightarrow\mathbb{C}$ is a function and $B$ is a non-empty finite subset of $A$, we write
$$
\E_{x\in B}f(x)=\frac{1}{|B|}\sum_{x\in B}f(x)
$$
as the average of $f$ on $B$. We would also abbreviate $\E_{x\in A}f(x)$ to $\E(f)$ if the supported set $A$ is finite and no confusion is caused.

We'll use Fourier analysis on $\Z$ with its dual $\T$. Let $f:\Z\to\C$ be a function, and define its Fourier transform by setting
\[
\hat{f}(\alpha)=\sum_xf(x)e(x\alpha),
\]
where $\alpha\in\T$. For functions $f,g:\Z\to\C$ define the convolution of $f$ and $g$ as
\[
f*g(x)=\sum_yf(y)g(x-y).
\]
Then, basic properties of Fourier analysis, for $f,g:\Z\to\C$
\begin{enumerate}
\item(Fourier inversion formula )	\quad $f(x)=\int_\T\hat f(\alpha)e(-x\alpha)\rd\alpha$;
\item(Parseval's identity)\quad   $\norm{f}_2=\|\hat f\|_2$;
\item \quad $\hat{f*g}=\hat f\cdot\hat g.$
\end{enumerate}

 $\eps>0$ is always  an arbitrarily small number, $c>0$ is a small number, and both $c$ and $\eps$ are allowed to change at different occurrences. For a function $f$ and positive-valued function $g$, write $f\ll g$ or $f=O(g)$ if there exists a constant $C>0$ such that $|f(x)|\leq Cg(x)$ for all $x$; and write $f\gg g$ if $f$ is also positive-valued and there is a constant $C>0$ such that $f(x)\geq Cg(x)$ for all $x$.

\subsection{Preliminaries}

Suppose that $x$ is a real number, $a$ and $q$ are positive  integers,  we write
\[
\psi(x;q,a)=\sum_{n\leq x\atop n\equiv a\pmod q}\Lambda(n).
\]
Estimating $\psi(x;q,a)$ is one of the central problems in analytic number theory, and the following lemma is from Hoheisel \cite{Hoh}.

\begin{lemma}[Primes in short arithmetic progressions]\label{short-ap}
Suppose that $1\leq a\leq q\leq(\log x)^A$ are coprime integers and $x^{7/12+\eps}\leq h\leq x$, then
\[
\sum_{x<n\leq x+h\atop n\equiv a\pmod q}\Lambda(n)=\frac{h}{\phi(q)}\Bigbrac{1+O_\eps\Bigbrac{\exp\bigbrac{-c\frac{\log^{1/3}x}{(\log\log x)^{1/3}}}}}.
\]	
\end{lemma}

When $h=x$, we can deal with exponential sums in arithmetic progressions (like (\ref{exp})) with $q$ beyond the limitation of Lemma \ref{short-ap}  with the help of the next lemma, which is, practically, a modification of \cite[Proposition 4.7]{RS}. The proof can be found in Appendix A.

\begin{lemma}[Exceptional pair result] \label{exceptional}
Let $D_1\geq D_0\geq2$ be some parameters. There is an absolute constant $c_0>0$, a primitive character $\chi$ modulus $q_0$ with $q_0\leq  D_1$, and a real number $\rho$ satisfying $(1-\rho)^{-1}\ll q_0^{1/2}(\log q_0)^2$ such that the following statement holds.

For all real $x\geq1$ and integers $1\leq a\leq q$ with $q\leq D_0$ and $(a,q)=1$, we have
\[
\psi(x;q,a)=\frac{\bar{\chi_0}(a)x}{\phi(q)}-E(x;q,a)+O\biggbrac{x\exp\Bigbrac{-c_0\frac{\log x}{\sqrt{\log x}+\log D_0}}(\log D_0)^2},
\] 
where
\begin{align}\label{eqa}
E(x;q,a)=\begin{cases}
\frac{\bar{\chi_0\chi}(a)x^\rho}{\phi(q)\rho}\qquad &\text{if } q_0|q;\\
0&\text{otherwise,}	
\end{cases}
\end{align}
and $\chi_0$ denotes the principal character of modulus $q$.
\end{lemma}

Given that our iteration process revolves around arithmetic progressions, we must confine the elements $p\in\mathcal A$ being in arithmetic progressions. To achieve this, let's introduce the majorant function $\Lambda_{b,d}:\Z\to\mathbb R_{\geq0}$ as follows:
\[
\Lambda_{b,d}(x)=\begin{cases}
\frac{\phi(d)}{d}\Lambda(b+dx)\quad &\text{ if }	1\leq x\leq X;\\
0&\text{ otherwise,}
\end{cases}
\]
which is essentially supported in the arithmetic progression $P=b+d\cdot[X]$ and $(b,d)=1$.
What follows is a simple consequence of the first conclusion of  Theorem \ref{thm2}.

\begin{lemma}[Alternative restriction]\label{restriction}
Let $0<\delta<1$ be a number and $X\gg N^{3/4}$. For every relatively prime pair $1 \leq b \leq d \leq \exp\left(\frac{c_0}{20k}\sqrt{\log X}\right)$, where $c_0 > 0$ is defined as in Lemma \ref{exceptional},  every function $\nu:[X]\to\C$  satisfying $|\nu|\leq\Lambda_{b,d}$ pointwise, and every real number $p>2$, we  have
\[
\|\hat f\|_p^p=\int_\T|\hat f(\alpha)|^p\rd\alpha\ll_p X^{p-1},
\]
where $f=\nu-\delta\Lambda_{b,d}$.
\end{lemma}

\begin{proof}

Practically, suppose that $p>2$ is a real number, on recalling  $f=\nu-\delta\Lambda_{b,d}$,  the use of mean value inequality yields that
\begin{multline*}
\int_\T|\hat f(\alpha)|^p\rd \alpha=\int_\T\Bigabs{\hat{\nu}(\alpha)-\delta\hat\Lambda_{b,d}(\alpha)}^p\rd\alpha
\ll_p\int_\T|\hat\nu(\alpha)|^p\rd\alpha+\delta^p\int_\T|\hat\Lambda_{b,d}(\alpha)|^p\rd\alpha.
\end{multline*}
Then it follows from the application of Theorem \ref{thm2} (1) to functions $\nu$ and $\Lambda_{b,d}$ separately, together with the assumption $0<\delta<1$.
 
\end{proof}

\section{Exponential sum estimates}

In this section, we'll do some preparations and the major object is to estimate the auxiliary exponential summation (\ref{exp}), that is,
\[
	S_d(\alpha)=\sum_{y\in[M]}\frac{\phi(d)}{d}ky^{k-1}\Lambda(dy+1)e(y^k\alpha),
\]
where $k\geq1$ is an integer, $N'$ is a large number, $1\leq d\leq\exp(\frac{c_0}{20k}\sqrt{\log N'})$,  $M=\floor{N'^{1/k}}$ and $\alpha\in\T$, where $c_0$ is as in Lemma \ref{exceptional}. Here, in order to avoid any potential confusion (in Section 5) we changed the letter $N$ in (\ref{exp}) to $N'$. 

In the light of the Hardy-Littlewood method we'll partition frequencies $\alpha\in\T$ into major arcs and minor arcs, and then calculate the behavior of $S_d(\alpha)$ in major and minor arcs respectively. We now handle these tasks by turns.  Let us divide the interval $\T$ into major arcs $\major$ and minor arcs $\minor$ which are defined in the following forms
\begin{align}\label{major-arc}
\major&=\bigcup_{q\leq Q}\bigset{\alpha\in\T:\norm{q\alpha}\ll Q/N'},\\	
\minor&=\T\backslash\major\nonumber,
\end{align}
where
\begin{align}\label{deq}
Q=(\log N')^{2^{2k}D}, 
\end{align}
and $D>10k^2$ is a sufficiently large number.

We shall calculate the asymptotic formula of the exponential sum $S_d(\alpha)$ in major arcs first. But before it, an auxiliary lemma.

\begin{lemma}\label{cq}
Let $k\geq1$ be an integer. Let $a$ and $q$ be relatively prime positive integers. Then for any coprime pair $b,t\geq1$ we have
\[
\left|\sum_{r(q)\atop(tr+b,q)=1}e\bigbrac{ar^k/q}\right|\ll q^{1-\frac{1}{k}+\eps},
\]
where $\eps>0$ is arbitrarily small.

\begin{proof}
We first claim that the left-hand side expression is multiplicative in the parameter $q$. To see this, let $q=q_1q_2$ with $(q_1,q_2)=1$, and $r=r_1q_2+r_2q_1$, on noting that
\[
(tr+b,q)=(t r_1q_2+tr_2q_1+b,q_1q_2)=(tr_1q_2+b,q_1)(tr_2q_1+b,q_2),
\]
we then have
\[
\sum_{r(q_1q_2)\atop(tr+b,q_1q_2)=1}e\biggbrac{\frac{ar^k}{q_1q_2}}=\sum_{r_1(q_1)\atop(tr_1q_2+b,q_1)=1}e\biggbrac{\frac{ar_1^kq_2^{k-1}}{q_1}}\sum_{r_2(q_2)\atop(tr_2q_1+b,q_2)=1}e\biggbrac{\frac{ar_2^kq_1^{k-1}}{q_2}}.
\]
Thus, without loss of generality, we may assume that $q=p^m$ with $p\in\P$. If $p\arrowvert t$, then for any $r$ modulus $p^m$, $(tr+b,p)=1$ always holds, just by taking note that $(b,t)=1$. Thus, in this case it follows from \cite[Theorem 4.2]{Va} that
\[
\left|\sum_{r(p^m)\atop(tr+b,p)=1}e\biggbrac{\frac{ar^k}{p^m}}\right|=\Bigabs{\sum_{r(p^m)}e\biggbrac{\frac{ar^k}{p^m}}}\ll (p^m)^{1-\frac{1}{k}+\eps}.
\]

Now we assume that $(p,t)=1$. By taking advantage of the M\"obius function, one has
\begin{align*}
\sum_{r(p^m)\atop(tr+b,p)=1}e\biggbrac{\frac{ar^k}{p^m}}&=\sum_{r(p^m)}	e\biggbrac{\frac{ar^k}{p^m}}\sum_{d|(tr+b,p)}\mu(d)=\sum_{d|p}\mu(d)\sum_{r(p^m)\atop d|(tr+b)}e\biggbrac{\frac{ar^k}{p^m}}.
\end{align*}
Just from the definition, above M\"obius function $\mu$ will vanish except when $d=1$ and $d=p$,  and, thus, above expression is indeed
\[
\sum_{r(p^m)}e\biggbrac{\frac{ar^k}{p^m}}-\sum_{r(p^m)\atop tr+b\equiv0\pmod p}e\biggbrac{\frac{ar^k}{p^m}}.
\]
By making use of \cite[Theorem 4.2]{Va} once again one has,
\begin{align*}
\biggabs{\sum_{r(p^m)}e\biggbrac{\frac{ar^k}{p^m}}}\ll (p^m)^{1-\frac{1}{k}+\eps}.	
\end{align*}
 As for the second term, taking note that $(t,p)=(t,b)=1$, let $t_p$ be the unique solution  to the linear congruence equation $tr+b\equiv0\pmod p$ in the range $[1,p]$, then for $r\in[p^m]$, the solutions are in the form of $r=t_p+sp$ with $s\in[p^{m-1}]$. And therefore,

\begin{multline*}
\sum_{r(p^m)\atop tr+b\equiv0\pmod p}e\biggbrac{\frac{ar^k}{p^m}}=\sum_{s(p^{m-1})}e\biggbrac{\frac{a(t_p+sp)^k}{p^m}}\\
=e\Bigbrac{\frac{at_p^k}{p^m}}\sum_{s(p^{m-1})}e\biggbrac{\frac{akt_p^{k-1}s+\cdots+ap^{k-1}s^k}{p^{m-1}}}.
\end{multline*}
Now let $p^\tau\parallel k$, if $\tau\geq\ceil{\frac{m}{2}}$, then $\sum_{s(p^{m-1})}e\biggbrac{\frac{akt_p^{k-1}s+\cdots+ap^{k-1}s^k}{p^{m-1}}}$  is trivially bounded by $O(1)$; if $0<\tau<\ceil{\frac{m}{2}}$, let $0<u\leq \tau$ be the integer such that $\gcd\Bigbrac{at_p^{k-1}k,at_p^{k-2}p\tbinom{k}{2},\cdots,ap^{k-1},p^{m-1}}= p^u $, it can be deduced from \cite[Theorem 7.1]{Va} that
\begin{multline*}
\biggabs{\sum_{s(p^{m-1})}e\biggbrac{\frac{akt_p^{k-1}s+\cdots+ap^{k-1}s^k}{p^{m-1}}}}\leq k\biggabs{\sum_{s(p^{m-u-1})}e\biggbrac{\frac{a_1's+\cdots+a_k's^k}{p^{m-u-1}}}}\\\ll(p^{m-u-1})^{1-\frac{1}{k}+\eps},
\end{multline*}
where $\gcd(a_1',\cdots,a_k',p)=1$; if $\tau=0$, it is directly from \cite[Theorem 7.1]{Va} that
\[
\left|\sum_{r(p^m)\atop tr+b\equiv0\pmod p}e\biggbrac{\frac{ar^k}{p^m}}\right|=\biggabs{\sum_{s(p^{m-1})}e\biggbrac{\frac{akt_p^{k-1}s+\cdots+ap^{k-1}s^k}{p^{m-1}}}}\ll(p^{m-1})^{1-\frac{1}{k}+\eps}.
\]
Putting all above cases together we obtain
\[
\left|\sum_{r(p^m)\atop (tr+b,p)=1}e\biggbrac{\frac{ar^k}{p^m}}\right|\ll(p^m)^{1-\frac{1}{k}+\eps},
\]
and the lemma follows.
\end{proof}

\end{lemma}

\begin{lemma}[Major-arc behavior of $S_d(\alpha)$]\label{major lemma}
Suppose that $\alpha\in\T$ can be written as $\alpha=\frac{a}{q}+\beta$ with coprime pair $1\leq a\leq q$. Let $0<c_0<1$ be  as in Lemma \ref{exceptional}, then
\[
S_d(\alpha)\ll q^{-\frac{1}{k}+\eps}N'(1+N'|\beta|)^{-1}+N'(1+N'|\beta|)\exp\bigbrac{-c\sqrt{\log N'}}
\]
holds for every sufficiently small $\eps>0$ and all integers $d,q\ll\exp\bigbrac{\frac{c_0}{20k}\sqrt{\log N'}}$.
\end{lemma}

\begin{proof}
In the first place, we consider the much easier case, that is $\alpha=a/q$. Expanding the definition of $S_d(\alpha)$ with $\alpha=a/q$, and then making  change of variables $dy+1\mapsto y$ we have
\begin{align*}
S_d(a/q)	&=\sum_{y\in[M]}k\frac{\phi(d)}{d}y^{k-1}\Lambda(dy+1)e\Bigbrac{\frac{ay^k}{q}}\\
&=\sum_{y\in[dM+1]\atop y\equiv1\pmod d} k\frac{\phi(d)}{d}\Bigbrac{\frac{y-1}{d}}^{k-1}\Lambda(y)e\Bigbrac{\frac{a}{q}\cdot\Bigbrac{\frac{y-1}{d}}^k}.
\end{align*}
We now split the above summation interval into sub-progressions according to the residue classes of $\frac{y-1}{d}$ modulo $q$, to obtain that
\[
S_d(a/q)=\sum_{r(q)}\threesum{y\in[dM+1]}{y\equiv1\pmod d}{\frac{y-1}{d}\equiv r\pmod q}k\frac{\phi(d)}{d}\Bigbrac{\frac{y-1}{d}}^{k-1}\Lambda(y)e\Bigbrac{\frac{a}{q}\cdot\Bigbrac{\frac{y-1}{d}}^k}.
\]
Then it is quite easy to see that
\begin{align}\label{saq}
S_d(a/q)=\sum_{r(q)}e\bigbrac{\frac{ar^k}{q}}\Psi(dM+1;dq,dr+1),
\end{align}
where
\[
\Psi(dM+1;dq,dr+1)=\sum_{y\in[dM+1]\atop y\equiv dr+1\pmod{dq}}\Lambda(y)k\frac{\phi(d)}{d}\Bigbrac{\frac{y-1}{d}}^{k-1}.
\]

Our task now is to compute $\Psi(dM+1;dq,dr+1)$, and the main tools are Lemma \ref{exceptional} and Abel's summation formula. In practice, an application of summation by parts yields that
\begin{multline}\label{3-2}
\Psi(dM+1;dq,dr+1)=\psi(dM+1;dq,dr+1)\frac{\phi(d)}{d}kM^{k-1}\\
-\frac{\phi(d)}{d}\int_1^{dM+1}\psi(t;dq,dr+1)\biggbrac{k\Bigbrac{\frac{t-1}{d}}^{k-1}}^\prime\rd t,
\end{multline}
on recalling that $\psi(x;q,a)=\sum_{n\in[x]\atop n\equiv a\pmod q}\Lambda(n)$.

Let   $D_0=\exp\bigbrac{\frac{c_0}{10k}\sqrt{\log N'}}$ and $D_1=D_0^2$, it then follows from Lemma \ref{exceptional} that when $d,q\leq\exp(\frac{c_0}{20k}\sqrt{\log N'})$ the  first term of the preceding expression equals to
\begin{align}\label{t-equal-dm}
\phi(d)kM^{k-1}\Bigbrac{\frac{\bar\chi_0(dr+1)M}{\phi(dq)}-\frac{E(dM;dq,dr+1)}{d}}+O\bigbrac{M^k\exp\bigbrac{-c\sqrt{\log N'}}}.	
\end{align}
Meanwhile, it can be also   seen from Lemma \ref{exceptional} that when $dq\leq D_0$ the  second summation in (\ref{3-2}) is equal to
\begin{multline}\label{t-leq-dm}
-\frac{\phi(d)}{d}\frac{\bar{\chi_0}(dr+1)}{\phi(dq)}\int_1^{dM+1}	t\Bigbrac{k\Bigbrac{\frac{t-1}{d}}^{k-1}}^\prime\rd t+\frac{\phi(d)}{d}\int_1^{dM+1}E(t;dq,dr+1)\Bigbrac{k\Bigbrac{\frac{t-1}{d}}^{k-1}}^\prime\rd t\\
+O\Bigbrac{d^{1-k}\log N'\int_1^{dM+1}t^{k-1}\exp\bigbrac{-c_0\frac{\log t}{\sqrt{\log t}+\frac{c_0}{10k}\sqrt{\log N'}}}\rd t}.
\end{multline}
O recalling that (\ref{eqa}), we notice that when $q_0|dq$ the second term of (\ref{t-leq-dm}) equals to $\frac{\phi(d)}{d}$ times
\[
\frac{\bar{\chi\chi_0}(dr+1)}{\phi(dq)}d^{1-k}\int^{dM+1}_1\frac{t^\rho}{\rho}k(k-1)(t-1)^{k-2}\rd t=\frac{\bar{\chi\chi_0}(dr+1)(dM)^\rho}{\phi(dq)\rho}\frac{k(k-1)}{k-1+\rho}M^{k-1}+O(d^{1-k});
\]
otherwise, this term vanishes. On combining the above two cases, and recalling (\ref{eqa}) once again, we can conclude that the second term of (\ref{t-leq-dm}) is always dominant by
\[
\frac{\phi(d)}{d}E(dM;dq,dr+1)\frac{k(k-1)}{k-1+\rho}M^{k-1}.
\]
As for the big O-term (the third term) of (\ref{t-leq-dm}), on noticing that 
\[
\frac{c_0\log t}{\sqrt{\log t}+\frac{c_0}{10k}\sqrt{\log N'}}\geq\frac{\frac{c_0}{2k}\log N'}{\sqrt{\log (dM)}+\frac{c_0}{10k}\sqrt{\log N'}}\geq\frac{c_0}{4}\sqrt{\log N'}
\]
whenever  $M^{1/2}\leq t\leq dM+1$,  thus, this term can be bounded up by
\[
d^{1-k}\log N'\biggset{\int_1^{M^{1/2}}t^{k-1}\rd t+\exp\bigbrac{-\frac{c_0}{4}\sqrt{\log N'}}\int_{M^{1/2}}^{dM+1}t^{k-1}\rd t}\ll M^k\exp(-c\sqrt{\log N'}),
\]
as $d\log N'\exp\bigbrac{-\frac{c_0}{4}\sqrt{\log N'}}\ll\exp\bigbrac{-c\sqrt{\log N'}}$ whenever $d\ll\exp\bigbrac{\frac{c_0}{20k}\sqrt{\log N'}}$.
 Therefore, (\ref{t-leq-dm}) equals to
\[
-\frac{\phi(d)\bar{\chi_0}(dr+1)}{\phi(dq)}(k-1)M^k+\frac{\phi(d)}{d}E(dM;dq,dr+1)\frac{k(k-1)}{k-1+\rho}M^{k-1}+O(M^k\exp(-c\sqrt{\log N'})).
\]
We now substitute the above expression and (\ref{t-equal-dm}) into (\ref{3-2}), and then substitute it into (\ref{saq}) to conclude that   when $q_0|dq$,
\begin{multline*}
S_d(a/q)=\frac{\phi(d)}{\phi(dq)}C(q;a,d)\biggset{M^k-\bar{\chi}(dr+1)\frac{kd^{\rho-1}}{k-1+\rho}M^{k-1+\rho}}\\+O\bigbrac{M^k\exp(-c\sqrt{\log N'})};
\end{multline*}
and when $q_0 \nmid dq$,
\[
S_d(a/q)=\frac{\phi(d)}{\phi(dq)}C(q;a,d)M^k+O\bigbrac{M^k\exp(-c\sqrt{\log N'})},
\]
where $C(q;a,d)=\sum_{r(q)\atop (dr+1,q)=1}e\bigbrac{\frac{ar^k}{q}}$.
Now let $\alpha=a/q+\beta$, it follows from Abel's summation formula, as well as $N' \asymp M^k$, that
\begin{multline*}
S_d(\alpha)	=\sum_{y\in[M]}k\frac{\phi(d)}{d}y^{k-1}\Lambda(dy+1)e\biggbrac{\frac{ay^k}{q}}e(\beta y^k)\\
=S_d(a/q)e(\beta M^k)-\int_1^M\sum_{y\in[t]}k\frac{\phi(d)}{d}y^{k-1}\Lambda(dy+1)e\biggbrac{\frac{ay^k}{q}}\Bigbrac{e(\beta t^k)}^\prime\rd t.
\end{multline*}
 From some computation which is similar to the previous paragraph it turns out that when $q_0|dq$,
\begin{multline}\label{asymp-1}
S_d(\alpha)	=	\frac{\phi(d)}{\phi(dq)}C(q;a,d)\biggset{\int_1^{M^k}e(\beta t)\rd t-\bar\chi(dr+1)d^{\rho-1}\int_1^{M^k}t^{\frac{\rho-1}{k}}e(\beta t)\rd t}\\+O\biggbrac{N'(1+N'|\beta|)\exp\bigbrac{-c\sqrt{\log N'}}};
\end{multline}
and when $q_0\nmid dq$,
\begin{multline}\label{asymp-2}
S_d(\alpha)	=	\frac{\phi(d)}{\phi(dq)}C(q;a,d)\int_1^{M^k}e(\beta t)\rd t+O\biggbrac{N'(1+N'|\beta|)\exp\bigbrac{-c\sqrt{\log N'}}}.
\end{multline}
If the case $q_0|dq$ arises, as $1/2<\rho<1$ is a real number that satisfies $1-\rho\gg D_0^{-1/2}(\log D_0)^{-2}\gg\exp(-c\sqrt{\log N'})$, it is clearly that the second term is rather smaller than the first one. Therefore, we can conclude that in both cases
\[
|S_d(\alpha)|\ll\frac{\phi(d)}{\phi(dq)}|C(q;a,d)| N'(1+N'|\beta|)^{-1}+N'(1+N'|\beta|)\exp\bigbrac{-c\sqrt{\log N'}},
\]
on noting that $\int_1^{M^k}e(\beta t)\rd t\ll\min\bigset{M^k,|\beta|^{-1}}\ll N'(1+N'|\beta|)^{-1}$.
We'll estimate the multiple factors in the first term in turns. From elementary properties of Euler totient function $\phi$, see \cite[ Theorem 2.5]{Apo} as an example, one has,
\[
\frac{\phi(d)}{\phi(dq)}=\frac{\phi(d)}{\phi(d)\phi(q)}\cdot\frac{\phi(\gcd(d,q))}{\gcd(d,q)}\leq\phi(q)^{-1}.
\]
Whilst, it  can be verified by Lemma \ref{cq} with $(a,q)=(d,1)=1$ that
\[
C(q;a,d)\ll q^{1-\frac{1}{k}+\eps},
\] 
for arbitrary small $\eps>0$. Therefore, the desired result follows from combining the above three representations together. 
 
\end{proof}

Next conclusion is the minor-arc estimates. Such kind of minor-arc analysis for exponential sum subjecting to a congruence equation has been developed in \cite{Gr05a}, the linear cases, and \cite{Chow}, the higher degree cases. We would like to refer interested readers to Appendix D for detailed proof.

\begin{lemma}[Minor arcs estimates]\label{minorsd} 
Assume that the minor arcs $\minor$ are  as in  (\ref{major-arc}), when $1\leq d\ll\exp(c\sqrt{\log N})$ for some constant $c>0$ then we have
\[
\sup_{\alpha\in\minor}|S_d(\alpha)|\ll N'(\log N')^{-D}.
\]	
\end{lemma}

\section{Restriction for shifted prime powers}

The main business of this section is to prove the second conclusion of Theorem \ref{thm2}, and the proof of the first part remains in Appendix B. We now introduce an auxiliary set: for any real number $0<\eta\ll1$, define the \textit{$\eta$-large spectrum of $S_d$} to be
\[
\R_\eta=\bigset{\alpha\in\T:|S_d(\alpha)|\geq\eta N},
\]
which is, in fact,  a set that contains all of the frequencies with large Fourier coefficients of $S_d$. The key observation is that the large spectrum set can not contain many elements.

\begin{lemma}[Large spectrum set is small]\label{spec}
For every real number $0<\eta\ll1$,
\[
\text{meas}\R_\eta\ll_\eps\eta^{-k(k+1)-2-\eps}N^{-1}
\]
holds for arbitrarily small $\eps=\eps(\eta)>0$.	
\end{lemma}

We can now prove Theorem \ref{thm2} (2) driving from Lemma \ref{spec}.

\vspace{2mm}

\noindent\emph{Proof of Theorem \ref{thm2} (2).}

\vspace{2mm}
Suppose that $p> k(k+1)+2$ is a real number, using dyadic argument, we have, for some absolute constant $C>0$,
\begin{multline*}
 \int_\T|S_d(\alpha)|^p\rd\alpha\leq\sum_{j\geq -C}\int_{\bigset{\alpha\in\T:\frac{N}{2^{j+1}}\leq|S_d(\alpha)|\leq\frac{N}{2^j}}}|S_d(\alpha)|^p\rd\alpha\\
 \leq\sum_{j\geq-C}2^{-jp}N^p\text{meas}\bigset{\alpha\in\T:|S_d(\alpha)|\geq 2^{-j-1}N}.
\end{multline*}
It then follows from Lemma \ref{spec} that
\[
 \int_\T|S_d(\alpha)|^p\rd\alpha\ll_\eps N^{p-1}2^{k(k+1)+2+\eps}\sum_{j\geq-C}2^{j(k(k+1)+2+\eps-p)}.
\]
And then, on noting the assumption that $p> k(k+1)+2$, it is permissible to take $\eps=\eps(p)$ sufficiently small so that above geometric series converges. As a consequence, it is immediately that the left-hand side integral is bounded by $O_p(N^{p-1})$.

\qed

In the rest of this section, we'll concentrate on proving Lemma \ref{spec}. Firstly, when $\eta$ is sufficiently small, Lemma \ref{spec} would follow from a mean value result.

 \begin{lemma}[A variant of Hua's lemma]\label{hua}
Let $p\geq k(k+1)+2$ be a positive number, we have
\[
\int_\T|S_d(\alpha)|^p\rd\alpha\ll N^{p-1}(\log N)^{k(k+1)+2}.
\]	
\end{lemma}
\begin{proof}
Firstly, by expanding the $p$-th powers, we notice that when $p=k(k+1)+2$
\begin{multline*}
\int_\T|S_d(\alpha)|^{k(k+1)+2}\rd\alpha=\int_\T\Bigabs{\sum_{y\in[M]} \frac{\phi(d)}{d}ky^{k-1}\Lambda(dy+1)e(\alpha y^k)}^{k(k+1)+2}\rd\alpha\\
\ll\sum_{y_1,\dots,y_{k(k+1)+2}\in[M]}\prod_{1\leq i\leq k(k+1)+2}\bigbrac{y_i^{k-1}\log(dy_i+1)}\int_\T e\bigbrac{\alpha\bigbrac{y_1^k+\cdots-y_{k(k+1)+2}^k}}\rd\alpha.
\end{multline*}
From the trivial upper bounds $y_i^{k-1}\log(dy_i+1)\ll M^{k-1}\log N$ for all $i=1,\dots,k(k+1)+2$, one may find  that
\[
\int_\T|S_d(\alpha)|^{k(k+1)+2}\rd\alpha\ll\bigbrac{M^{k-1}\log N}^{k(k+1)+2}\int_\T\Bigabs{\sum_{y\in[M]}e(y^k\alpha)}^{k(k+1)+2}\rd\alpha.
\]

Besides, let $\underline\alpha=(\alpha_1,\cdots,\alpha_k)\in\T^k$ be a $k$-tuple of reals, and define an exponential sum 
\[
F(\underline{\alpha};M)=\sum_{y\in[M]}e(y\alpha_1+\cdots+y^k\alpha_k).
\]
From the observation that $\int_{\T^k}|F(\underline{\alpha};M)|^{k(k+1)+2}\rd\underline\alpha$  counts the number of integral solutions to the system of  $k$ equations
\[
y_1^j+\cdots y_{1+\frac{k(k+1)}{2}}^j=y_{2+\frac{k(k+1)}{2}}^j+\cdots +y_{k(k+1)+2}^j\qquad 1\leq j\leq k
\]
with variables $y_i\in[M]\bigbrac{1\leq i\leq k(k+1)+2}$, hence
\begin{multline*}
\int_\T\biggabs{\sum_{y\in[M]}e(y^k\alpha)}^{k(k+1)+2}\rd\alpha\leq\sum_{|h_j|\leq k(k+1)M^j\atop1\leq j\leq k-1}\int_{\T^k}|F(\underline{\alpha};M)|^{k(k+1)+2}e\bigbrac{-h_1\alpha_1-\cdots-h_{k-1}\alpha_{k-1}}\rd\underline\alpha.
\end{multline*}
Making use of \cite[formula (7)]{BDG} with $p=k(k+1)+2$ together with the triangle inequality leads us to
\[
\int_\T\Bigabs{\sum_{y\in[M]}e(y^k\alpha)}^{k(k+1)+2}\rd\alpha\ll M^{k(k+1)+2-k}.
\]
Next, substitute the above inequality into the upper bound of $\int_\T|S_d(\alpha)|^{k(k+1)+2}\rd\alpha$ to get that
\[
\int_\T|S_d(\alpha)|^{k(k+1)+2}\rd\alpha\ll N^{k(k+1)+1}(\log N)^{k(k+1)+2},
\]
just on taking note that $N\geq M^k$.

On the other hand, it is fairly straightforward to verify that
\[
\norm{S_d}_\infty\leq\sum_{y\in[M]}k\frac{\phi(d)}{d}y^{k-1}\Lambda(dy+1)\ll N,
\]
once again recalling that $N\geq M^k$. We then can jump to conclude that
\[
\int_\T|S_d(\alpha)|^p\rd\alpha\leq\norm{S_d}_\infty^{p-k(k+1)-2}\int_\T|S_d(\alpha)|^{k(k+1)+2}\rd\alpha\ll N^{p-1}(\log N)^{k(k+1)+2}
\]
whenever $p\geq k(k+1)+2$.	
\end{proof}

\vspace{2mm}

\noindent\emph{Proof of Lemma \ref{spec}}.

\vspace{2mm}

 When $\eta\ll(\log N)^{-\frac{k(k+1)+2}{\eps}}$ Lemma \ref{spec} is just a corollary of Lemma \ref{hua}. Indeed, on recalling the definition of the large spectrum set $\R_\eta$,  Lemma \ref{hua} leads us to
\[
(\eta N)^{k(k+1)+2}\\text{meas}\R_\eta\leq\int_\T|S_d(\alpha)|^{k(k+1)+2}\rd\alpha\ll N^{k(k+1)+1}(\log N)^{k(k+1)+2},
\]
which gives
\[
\text{meas}\R_\eta\ll \eta^{-k(k+1)-2}(\log N)^{k(k+1)+2}N^{-1}.
\]
This allows us to conclude that $\text{meas}\R_\eta\ll \eta^{-k(k+1)-2-\eps}N^{-1}$ whenever $\eta\ll(\log N)^{-\frac{k(k+1)+2}{\eps}}$. Thus, in the remainder part of the proof we can reduce the consideration to those $\eta$ in the region
\begin{align}\label{eta}
(\log N)^{-\frac{k(k+1)+2}{\eps}}\ll\eta\ll1.	
\end{align}
 
 We may pick up a discrete sequence $\set{\alpha_1,\cdots,\alpha_R}\subset\T$ for which satisfies
\begin{enumerate}

\item For any $\alpha_r\in\set{\alpha_1,\cdots,\alpha_R}$, $|S_d(\alpha_r)|\geq\eta N$;
\item For any pair $\alpha_i\neq\alpha_j\in\set{\alpha_1,\cdots,\alpha_R}$, $|\alpha_i-\alpha_j|\geq N^{-1}$;
\item For any $\alpha\in\R_\eta$, there is some $\alpha_r\in\set{\alpha_1,\cdots,\alpha_R}$ such that $|\alpha-\alpha_r|\leq N^{-1}$;
\item $R$ is the largest integer such that above three events hold.

\end{enumerate}
Consequently, it suffices to show that 
\begin{align}\label{rr}
	R\ll_{\eps}\eta^{-k(k+1)-2-\eps}.
\end{align}
And, clearly, $\text{meas} \R_\eta\leq2R/N$. 

For $1\leq r\leq R$, let $c_r\in\set{z\in\C:|z|=1}$ be a number such that $c_rS_d(\alpha_r)=|S_d(\alpha_r)|$. One may find from the first assumption of the above  sequence $\set{\alpha_1,\cdots,\alpha_R}$ that
\[
R^2\eta^2N^2\leq\Bigbrac{\sum_{1\leq r\leq R}|S_d(\alpha_r)|}^2.
\]
Expanding  the definition of the exponential sum and using Cauchy-Schwarz inequality successively to obtain that
\begin{align*}
R^2\eta^2N^2&\leq\biggabs{\sum_{1\leq r\leq R}c_r\sum_{y\in[M]}k\frac{\phi(d)}{d}y^{k-1}\Lambda(dy+1)e(\alpha_ry^k)}^2\\
&\ll N\sum_{y\in[M]}k\frac{\phi(d)}{d}y^{k-1}\Lambda(dy+1)\Bigabs{\sum_{1\leq r\leq R}c_re(\alpha_ry^k)}^2	,
\end{align*}
just on noting that $\sum_{y\in[M]}k\frac{\phi(d)}{d}y^{k-1}\Lambda(dy+1)\ll N$. Expanding the square to double the variables $c_r$ we then have
\[
R^2\eta^2N\ll\sum_{y\in[M]}\sum_{1\leq r,r'\leq R}c_r\bar{c_{r'}}k\frac{\phi(d)}{d}y^{k-1}\Lambda(dy+1)e\bigbrac{(\alpha_r-\alpha_{r'})y^k}.
\]
In view of $c_r$, $c_{r'}$ are 1-bounded, we can swap the order of  summations and then take absolute value inside $r$ and $r'$, it then follows from the definition of $S_d(\alpha)$ that
\[
R^2\eta^2N\ll\sum_{1\leq r,r'\leq R}|S_d(\alpha_r-\alpha_{r'})|.
\] 
Let $\gamma>k$ be a positive number, utilizing H\"older's inequality one has
\[
R^2\eta^{2\gamma}N^\gamma\ll\sum_{1\leq r,r'\leq R}|S_d(\alpha_r-\alpha_{r'})|^\gamma.
\]
Taking note that $(\log N)^{-D/2}\ll\eta$ whenever $\eta$ satisfies (\ref{eta}), Lemma \ref{minorsd} with $N'=N$ leads us to that the contribution of $S_d(\alpha_r-\alpha_{r'})$ is negligible if $\alpha_r-\alpha_{r'}\in\minor$. We then can assume further that $\alpha_r-\alpha_{r'}\in\major$. Under such circumstances, the major-arc approximate recorded in Lemma \ref{major lemma} with $N'=N$ leads us to
\[
S_d(\alpha_r-\alpha_{r'})\ll q^{-\frac{1}{k}+\eps}N(1+N|\alpha_r-\alpha_{r'}-a/q|)^{-1}
\]
for some relatively prime pair $1\leq a\leq q\leq Q$(recalling the definition of major arcs $\major$). We also notice that when $q\geq \widetilde Q$,
\[
S_d(\alpha_r-\alpha_{r'})\ll \widetilde Q^{-\frac{1}{k}+\eps}N.
\]
Hence, if we take $\widetilde Q=\eta^{-2k+\eps}$ the contribution of $S_d(\alpha_r-\alpha_{r'})$ with $q>\tilde Q$ is also negligible. Combining what we have so far it is allowed to restrict our considerations to those $q$ which are no more than $\widetilde Q$. And we'll have
\[
\eta^{2\gamma}R^2\ll\sum_{q\leq\widetilde Q}\sum_{a(q)\atop(a,q)=1}\sum_{1\leq r,r'\leq R}q^{-\frac{\gamma}{k}+\eps}(1+N|\alpha_r-\alpha_{r'}-a/q|)^{-\gamma}\leq\sum_{1\leq r,r'\leq R}G(\alpha_r-\alpha_{r'}),
\]
where
\[
G(\alpha)=\sum_{q\leq\widetilde Q}\sum_{a(q)}q^{-\frac{\gamma}{k}+\eps}F(\alpha-a/q)
\]
and
\[
F(\beta)=(1+N|\sin\beta|)^{-\gamma}.
\]
We are now in the position of \cite[Eq.(4.16)]{Bou89}, just need to replace $N^2$ by $N$, replace $\delta$ by $\eta$. Following from Bourgain's proof, one may find that when $\gamma\geq k+\eps$
\[
\eta^{2\gamma}R^2 N^{-2}\ll N^{-1}\Bigset{\widetilde Q^\tau R N^{-1}+R^2N^{-2}\widetilde Q\cdot\#\set{|u|\leq N:d(\widetilde Q;u)\geq\widetilde Q^\tau}},
\]
where $\tau>0$ is a number to be chosen later, $d(\widetilde Q;u)$ counts the number of the divisors of $u$ which are smaller than $\widetilde Q$. It follows from \cite[Lemma 4.28]{Bou89} that
\[
R\ll  \eta^{-2\gamma}\widetilde Q^\tau+ \eta^{-2\gamma}\widetilde Q^{1-B}R
\]
with $\tau< B<+\infty$. By careful choice of $\tau$ and $B$, together with letting $\gamma=k+\eps$ will lead us to the desired bound (\ref{rr}) of $R$. 
It's worth noting that, even when $\eta$ is within the region (\ref{eta}), the estimate (\ref{rr}) can be improved to $R \ll \eta^{-2k-\epsilon}$. However,  for smaller values of $\eta$, Lemma \ref{hua} indeed plays a crucial role in pushing the exponent of $\eta^{-1}$ to reach $k(k+1)+2+\epsilon$.
\qed

\section{A Local Inverse Theorem}

The goal of this section is to prove the following local inverse result. Despite being in Fourier-analysis  language, the point is still that if the sum of a function $f$ in an arithmetic progression is zero, and the counting of configurations weighted by $f$ is non-zero, then either the common difference of this progression is too big or else there is a reasonable bias of $f$ towards a structural sub-region of the summation interval.

\begin{proposition}[Local inverse theorem]\label{local inverse}
Suppose that $\frac{1}{\log N}<\delta<1$ and $( a', q')=1$ are  coprime integers. Let  $\Lambda_{a',q'}:[N']\to\R_{\geq0}$ be the majorant function defined in (\ref{laq}),  $a'+q'\cdot[N']\subseteq[N]$ and  $N'\gg N^{7/12+\eps}$. Let $k\geq1$ and $1\leq d\ll\exp\bigbrac{\frac{c_0}{20k}\sqrt{\log N}}$  and $M=\floor{N'^{1/k}}$.  Assume that $\nu,\nu_1:[N']\to\mathbb C$ are functions satisfying $0\leq\nu,\nu_1\leq\Lambda_{a',q'}$ pointwise, and let $f=\nu_1-\delta\Lambda_{a',q'}$ be a function satisfying $\E(f)=0$.  If
\[
\delta^2 N'^2\ll\int_\T\bigabs{\hat f(\alpha)S_d(\alpha)\hat \nu(\alpha)}\rd\alpha,
\]
where $S_d(\alpha)=\sum_{y\in[M]}\frac{\phi(d)}{d}ky^{k-1}\Lambda(dy+1)e(y^k\alpha)$ is defined in (\ref{exp}),	 then one of the following two alternatives holds:
\begin{enumerate}
\item(Common difference is large)  $ q'\delta^{-2k^3-10k^2}\gg\log^A N$; 
\item($f$ has local structure) there is an arithmetic progression $P=a+q^k\cdot[X] \subseteq[N']$ with $( a, q)=1$,  $q\ll\delta^{-2k^3-6k^2}$, and $X\gg q^{1-k}(\log N)^{-2^{2k}D}N'$ and some constant $c>0$ such that
\[
\sum_{n\in P}f(n)\geq c\delta^{4k^3+20k^2+3}\sum_{n\in P}\Lambda(n).
\]
\end{enumerate}

\end{proposition}

By decomposing the integral $\T$ into major and minor arcs $\major$ and $\minor$ as in (\ref{major-arc}), it can be seen from the assumption of Proposition \ref{local inverse} that
\begin{align}\label{partition}
\delta^2N'^2\ll\int_\major\bigabs{\hat f(\alpha)S_d(\alpha)\hat \nu(\alpha)}\rd\alpha+\int_\minor\bigabs{\hat f(\alpha)S_d(\alpha)\hat \nu(\alpha)}\rd\alpha.
\end{align}
 We'd like to show that, with the help of the restriction lemmas, the integral over the minor arcs is negligible, i.e. bounded by $c\delta^2N^2$ for some tiny enough $c>0$. By making use of H\"older's inequality, the second term in the right-hand side of (\ref{partition}) is bounded by
\begin{multline*}
\quad \Bigbrac{\int_\minor\abs{S_d(\alpha)}^{k(k+1)+4}\rd\alpha}^{\frac{1}{k(k+1)+4}}\Bigbrac{\int_\T\abs{\hat f(\alpha)}^{\frac{k(k+1)+4}{k(k+1)+3}}\abs{\hat \nu(\alpha)}^{\frac{k(k+1)+4}{k(k+1)+3}}\rd\alpha}^{\frac{k(k+1)+3}{k(k+1)+4}}\\
	\leq\sup_{\alpha\in\minor}\bigabs{S_d(\alpha)}^{\frac{1}{k(k+1)+4}}\Bigbrac{\int_\T\abs{S_d(\alpha)}^{k(k+1)+3}\rd\alpha}^{\frac{1}{k(k+1)+4}}\\\Bigbrac{\int_\T\abs{\hat f(\alpha)}^{\frac{2k(k+1)+8}{k(k+1)+3}}\rd\alpha}^{\frac{k(k+1)+3}{2k(k+1)+8}}\Bigbrac{\int_\T\abs{\hat \nu(\alpha)}^{\frac{2k(k+1)+8}{k(k+1)+3}}\rd\alpha}^{\frac{k(k+1)+3}{2k(k+1)+8}}.
\end{multline*}
The application of Lemma \ref{restriction} for $f$ and $p=\frac{2k(k+1)+8}{k(k+1)+3}$, Theorem \ref{thm2} for $\Lambda_{b,d}=\Lambda_{a',q'}$ and $p=\frac{2k(k+1)+8}{k(k+1)+3}$ and for $S_d(\alpha)$ and $p=k(k+1)+3$ yields that
\[
\int_\minor\bigabs{\hat f(\alpha)S_d(\alpha)\hat \nu(\alpha)}\rd\alpha\ll \sup_{\alpha\in\minor}\bigabs{S_d(\alpha)}^{\frac{1}{k(k+1)+4}} N'^{\frac{2k(k+1)+7}{k(k+1)+4}}.
\]
It then follows from Lemma \ref{minorsd}, the assumptions $D>10k^2$  and  $(\log N)^{-1}\ll\delta$ that
\[
\int_\minor\bigabs{\hat f(\alpha)S_d(\alpha)\hat \nu(\alpha)}\rd\alpha\ll (\log N)^{-\frac{D}{k(k+1)+4}}N'^2\ll \delta^2N'^2.
\]
Substituting the above inequality into (\ref{partition}), one thus has
\begin{align}\label{major}
	\int_\major\bigabs{\hat f(\alpha)S_d(\alpha)\hat \nu(\alpha)}\rd\alpha\gg \delta^2 N'^2.
\end{align}

The next step is to show that we can gain from (\ref{major}) a long arithmetic progression that possesses much information about the function $f$. To carry out this procedure we are now proving that the inequality (\ref{major}) can be reduced to a subset of major arcs $\major$, and the set only contains frequencies that are close to rationals with quite small denominators (depending only on $\delta$).

Let $q$ be a positive integer, define a set with respect to $q$ as
\[
\major(q)=\Bigset{\alpha\in\T:\exists1\leq a\leq q\text{ with }(a,q)=1 \text{ such that }\bigabs{\alpha-\frac{a}{q}}\leq\frac{Q}{qN'}},
\]
where $Q$ is the number defined in (\ref{deq}). We now partition the major arcs $\major$  into two disjoint ranges
\[
\major_1=\bigcup_{1\leq q\ll\delta^{-2k^3-10k^2}}\major(q)\qquad\text{and}\qquad\major_2=\bigcup_{\delta^{-2k^3-10k^2}\ll q\ll Q}\major(q).
\]
It  follows from the triangle inequality and (\ref{major}) that
\begin{align}\label{sdt}
\delta^2 N'^2\ll\int_{\major_1}\bigabs{\hat f(\alpha)S_d(\alpha)\hat \nu(\alpha)}\rd\alpha+\int_{\major_2}\bigabs{\hat f(\alpha)S_d(\alpha)\hat \nu(\alpha)}\rd\alpha.
\end{align}
In an analogous way of proving that the integral over minor arcs $\minor$ is negligible, we can also show that the integral over $\major_2$ is negligible. Indeed, one can deduce from H\"older's inequality and restriction estimates, i.e. Lemma \ref{restriction} and Theorem \ref{thm2}, that
\[
\int_{\major_2}\bigabs{\hat f(\alpha)S_d(\alpha)\hat \nu(\alpha)}\rd\alpha\ll \sup_{\alpha\in{\major_2}}\bigabs{S_d(\alpha)}^{\frac{1}{k(k+1)+4}} N'^{\frac{2k(k+1)+7}{k(k+1)+4}}.
\]
On recalling that $q\ll Q\ll (\log N)^{2^{2k}D}$, the application of Lemma \ref{major lemma} shows that the above integral is bounded by
\[
 \Bigbrac{ q^{-\frac{1}{k}+\eps}\min\set{N',|\beta|^{-1}}+N'(1+N'|\beta|)\exp\bigbrac{-c\sqrt{\log N}}}^{\frac{1}{k(k+1)+4}}N'^{\frac{2k(k+1)+7}{k(k+1)+4}}
\]
for all $\frac{a}{q}+\beta\in\major_2$ with coprime pair $1\leq a\leq q$. The assumption $\delta^{-2k^3-10k^2}\ll q\ll(\log N)^{2^{2k}D}$ whenever $\alpha=\frac{a}{q}+\beta\in\major_2$ then gives us that
\[
\int_{\major_2}\bigabs{\hat f(\alpha)S_d(\alpha)\hat \nu(\alpha)}\rd\alpha\ll\delta^2 N'^2.
\]
Combining the above inequality with (\ref{sdt}) one thus has
\[
\int_{\major_1}\bigabs{\hat f(\alpha)S_d(\alpha)\hat \nu(\alpha)}\rd\alpha\gg\delta^2 N'^2.
\]
So far, all preparations have been completed.

\vspace{4mm}

\noindent\emph{Proof of Proposition \ref{local inverse}.}

\vspace{2mm}

The starting point of the proof is to observe that the definition of $\major_1$ and the above inequality shows that
\[
\delta^2 N'^2\ll\sum_{q\ll\delta^{-2k^3-10k^2}}\int_{\major(q)}|\hat{f}(\alpha)\hat \nu(\alpha)S_d(\alpha)|\rd \alpha.
\]
Taking note that the asymptotic formula in Lemma \ref{major lemma} is valid when $q\ll\delta^{-2k^3-10k^2}$ and $\delta\gg(\log N)^{-1}$, in consequence of the above inequality, H\"older's inequality and  Theorem \ref{thm2} and Lemma \ref{major lemma}, one has
\begin{align*}
\delta^2 N'^2&\ll\sum_{q\ll\delta^{-2k^3-10k^2}}\norm{\hat\nu}_3\Bigbrac{\int_{\major(q)}|\hat{f}(\alpha)|^2\rd\alpha}^{1/2}\Bigbrac{\int_{\major(q)}|S_d(\alpha)|^6\rd\alpha}^{1/6}\\
&\ll \sum_{q\ll\delta^{-2k^3-10k^2}}q^{-1/k+\eps}N'^{5/3}\Bigbrac{\int_{\major(q)}|\hat{f}(\alpha)|^2\rd\alpha}^{1/2}\Bigbrac{\int_{|\beta|\leq\frac{Q}{qN'}}(1+N'|\beta|)^{-6}\rd\beta}^{1/6}\\
&\ll \sum_{q\ll\delta^{-2k^3-10k^2}}q^{-1/k+\eps}  N'^{3/2}\frac{\phi(q)}{q}\Bigbrac{\int_{\major(q)}|\hat{f}(\alpha)|^2\rd\alpha}^{1/2}.
\end{align*}
On noting that $\sum_{q\le\delta^{-2k^3-10k^2}}q^{-\frac{1}{k}+\eps}\ll
\delta^{-2k^3-10k^2}$, clearly, the pigeonhole principle implies that  there is some $q\le\delta^{-2k^3-10k^2}$ such that
\begin{align}\label{mass}
\delta^{4k^3+20k^2+4}N'\ll \bigbrac{\frac{\phi(q)}{q}}^2\int_{\major(q)}|\hat{f}(\alpha)|^2\rd\alpha.
\end{align}
 
For an integer $0\leq r<q$, suppose that $P_r=r+q\cdot[X]$ is the residue class of  $r$  modulo $q$, where $q$ is the integer such that (\ref{mass}) holds and $X=Q^{-1}N'/(2\pi)$. We are going to show that, if $N'$  is not small, for any frequency $\alpha$ in $\major(q)$ and for any $0\leq r<q$, the Fourier coefficient $|\hat{1}_{-P_r}(\alpha)|$ is large. For this purpose, we now consider the Fourier transform of the indicator function $1_{-P_r}$. After changing variables $x\to -x$ one has
\[
|\hat{1}_{-P_r}(\alpha)|=\biggabs{\sum_x1_{-P_r}(x)e(x\alpha)}=\biggabs{\sum_x1_{P_r}(x)e(-x\alpha)}.
\]
On omitting the indicator function $1_{P_r}$, 
\begin{align}\label{hatp}
|\hat{1}_{-P_r}(\alpha)|=\biggabs{\sum_{x\in[X]}e(-\alpha r-qx\alpha)}=\biggabs{\sum_{x\in[X]}e(-qx\alpha)}.
\end{align}
With the help of the estimate $|1-e(\alpha)|\leq2\pi\norm{\alpha}$, as well as the triangle inequality, we can ensure that
\[
|\hat{1}_{-P_r}(\alpha)|\geq X-\sum_{x\in[X]}2\pi\norm{qx\alpha}\geq X-\pi X^2\norm{q\alpha}.
\]
It thus follows immediately from the facts  $\norm{q\alpha}\ll QN'^{-1}$ $(\alpha\in\major(q))$ and $X=Q^{-1}N'/(2\pi)$ that  whenever $\alpha\in\major(q)$,
\begin{align}\label{massp}
|\hat{1}_{-{P_r}}(\alpha)|\geq X/2	
\end{align}
  holds uniformly for all $r\,(\text{mod }q)$.

We now claim that whenever $N'$ is not small, there is some $r\,(\text{mod }q)$ and some point $x\in\Z$ such that $f*1_{-P_r}(x)$ is sufficiently large. Virtually, on combining (\ref{mass}) with (\ref{massp}), we have
\[
\delta^{4k^3+20k^2+4}\bigbrac{\frac{q}{\phi(q)}}^2X^2N'\ll\int_{\major(q)}|\hat{1}_{-P_r}(\alpha)|^2|\hat{f}(\alpha)|^2\rd\alpha.
\]
Then by removing the restriction of domain of the integral to get that
\[
\delta^{4k^3+20k^2+4}\bigbrac{\frac{q}{\phi(q)}}^2X^2N'\ll \int_\T|\hat{1}_{-P_r}(\alpha)|^2|\hat{f}(\alpha)|^2\rd\alpha.
\]
It can be deduced from  Parseval's identity that
\begin{align}\label{l2-bd}
\delta^{4k^3+20k^2+4}\bigbrac{\frac{q}{\phi(q)}}^2X^2N'\ll \norm{f*1_{-P_r}}_2^2\leq\norm{f*1_{-P_r}}_\infty \sum_x|f*1_{-P_r}(x)|.
\end{align}
The assumption $f\geq-\delta\Lambda_{a',q'}$ pointwise yields  that for each natural number $x$
\[
f*1_{-P_r}(x)\geq-\delta\frac{\phi(q')}{q'}\sum_{n\in x+P_r}\Lambda(a'+q'n)=-\delta\frac{\phi(q')}{q'}\sum_{n\in[X]}\Lambda(a'+q'(x+r)+qq'n).
\]
 As $\text{supp}(f)\subseteq[N']$, we can also suppose that  $1\leq a'+q'(x+r)\leq N-qq'X$. Since $qq'X\gg N^{7/12+\eps}$ and $q\ll\delta^{-2k^3-6k^2}\ll(\log N)^A$ whenever $\delta^{-1}\ll\log N$, it follows from Lemma \ref{short-ap} and the assumption $q'\ll(\log N)^A$ that 
\[
\sum_{N-qq'X<n\leq N\atop n\equiv a'+q'(x+r)\pmod{qq'}}\Lambda(n)\leq 2\frac{qq'X}{\phi(qq')}.	
\]
One thus has
\[
f*1_{-P_r}(x)\geq-2\delta\frac{\phi(q')}{q'} \frac{qq'X}{\phi(qq')}\geq-2\delta\frac{\phi(q')qX}{\phi(q)\phi(q')}=-2\delta\frac{qX}{\phi(q)}.
\]
We may firstly assume that there is some integer $x\in\Z$  such that  $|f*1_{-P_r}(x)|\geq3\delta \frac{qX}{\phi(q)}$. Given the above analysis, we must have, in this case, 
\begin{align}\label{case1}
f*1_{-P_r}(x)\geq 3\delta \frac{qX}{\phi(q)}.
\end{align}
Expanding the convolution one may conclude that 
\begin{align}\label{case-1}
\sum_{n\in r+x+q\cdot[X]}f(n)\geq 3\delta \frac{qX}{\phi(q)}.
\end{align}
 Thus, we may assume that $\norm{f*1_{-P_r}}_\infty<3\delta\frac{qX}{\phi(q)}$. In view of (\ref{l2-bd}), we have
\[
\sum_x|f*1_{-P_r}(x)|\gg \delta^{4k^3+20k^2+3}\frac{q}{\phi(q)}XN'.	
\]

Meanwhile, we also see from the definition of the convolution that 
\[
\sum_xf*1_{-P_r}(x)=\sum_yf(y)\sum_x1_{-P_r}(x-y).
\]
Take note that for each fixed number $y$, the inner sum is always $X$, the length of the progression $P_r$.  As a consequence of $\E(f)=0$, we also have
\[
\sum_xf*1_{-P_r}(x)=0.
\]
Hence, by combining the above two hands one has
\[
c\delta^{4k^3+20k^2+3}\frac{q}{\phi(q)}XN'\leq\sum_x\max\bigset{f*1_{-P_r}(x),0},
\]
it then follows from the pigeonhole principle, together with $\text{supp}(f)\subset[N']$, that there is  some $x$ such that
\[
f*1_{-r-q\cdot[X]}(x)\geq c\delta^{4k^3+20k^2+3}\frac{q}{\phi(q)}X.
\]
By expanding the convolution  and taking (\ref{case-1}) into account, there is always a progression $a+q\cdot[X]$  such that
\[
c\delta^{4k^3+20k^2+3}\frac{q}{\phi(q)}X\leq\sum_{n\in  a+q\cdot[X]}f(n).
\]
For some technical reason, we'd like to partition the above arithmetic progression into sub-progressions with common difference $q^k$, and as a consequence, each of the sub-progressions has length $\gg\floor{\frac{X}{q^{k-1}}}\gg q^{1-k}(\log N)^{-2^{2k}D}N'$. It is not hard to find from the above inequality and the pigeonhole principle that there is such a sub-progression $P'$ such that
\[
c\delta^{4k^3+20k^2+3}\frac{q}{\phi(q)}|P'|\leq\sum_{y\in P'}f(y).
\] 
For simplification of writing, let $P'=a'+q^k\cdot[\frac{X}{q^{k-1}}]$. Noting $|f|\ll\Lambda_{a',q'}$ pointwise, the above progression $P'$ must satisfy $(a',q)=1$. It then follows from Lemma \ref{short-ap} and the assumption  $q^k\ll (\log N)^A$ that
\[
\sum_{n\in P'}\Lambda(n)=\sum_{a'< n\leq a'+qX\atop n\equiv a'\pmod {q^k}}\Lambda(n) \asymp\frac{q}{\phi(q)}|P'|,
\] 
which yields that 
\[
c\delta^{4k^3+20k^2+3}\sum_{n\in a'+q^k\cdot[\ceil{\frac{X}{q^{k-1}}}]}\Lambda(n)\leq \sum_{n\in a'+q^k\cdot[\ceil{\frac{X}{q^{k-1}}}]}f(n).
\]
 The lemma follows.

\qed

\section{Density Increment}

We are going to prove Theorem \ref{main} in this section. The proof strategy is to run a density increment argument in primes in arithmetic progressions. Roth uses the corresponding argument in integers to bound the size of sets lacking 3-term arithmetic progressions, and  S\'ark\"ozy employs it to bound the size of sets lacking Furstenberg-S\'ark\"ozy configurations $x,x+y^2$. While the difficulty in dense and sparse sets may vary, the fundamental idea remains consistent: to find a fairly large sub-progression in which $\A$ has an increased relative density if $\A$ lacks the desired configurations. 

\begin{lemma}[Lacking configurations leads to increased density]\label{non-pseudorandomness}
Suppose that $\log ^{-1}N\ll\delta<1$ is a parameter. Suppose that $Q=a+q^k\cdot[N']\subseteq[N]$ is an arithmetic progression with  $(a,q)=1$ and $N'\gg N^{7/12+\eps}$. Let $\A\subseteq\P$ be a subset of primes without configurations $p_1,p_1+(p_2-1)^k$, where $p_1,p_2$ are primes, and 
\[
\sum_{n\in Q}\Lambda\cdot1_\A(n)=\delta\sum_{n\in Q}\Lambda(n).
\]
Then one of the following statements holds
\begin{enumerate}
\item (the common difference is large) $q^k\delta^{-2k^4-10k^3}\gg \log^A N$; or
\item ($\A$ has increasing density on a sub-progression)  there  is a sub-progression $Q'=a+a'q^k+(qq')^k\cdot[X]$ inside $Q$ with $(a',q')=1$, $q'\ll\delta^{-2k^3-10k^2}$ and $X\gg q'^{1-k}(\log N)^{-2^{2k}D}N'$  such that
\[
\sum_{n\in Q'}\Lambda\cdot1_\A(n)\geq\delta\Bigbrac{1+c\delta^{4k^3+22k^2}}\sum_{n\in Q'}\Lambda(n).
\]
\end{enumerate}
\end{lemma}

\begin{proof}
Set $\Lambda_{a,q^k}=\frac{\phi(q)}{q}\Lambda(a+q^k\cdot)$ and $1_{\A_{a,q^k}}=1_\A(a+q^k\cdot)$, and let $f=\Lambda_{a,q^k}\cdot 1_{\A_{a,q^k}}-\delta\Lambda_{a,q^k}:[N']\to\C$ be the  weighted balanced function. Clearly, from the assumption, the average of $f$ is zero, i.e. $\sum_{n\in[N']}f(n)=0$.

Inspired by the proof of \cite[Lemma 5]{BM}  we consider the following expression
\[
\int_\T\hat f(-\alpha)\hat{\Lambda_{a,q^k}\cdot 1_{\A_{a,q^k}}}(\alpha)S_q(\alpha)\rd\alpha,
\]
where $S_q(\alpha)=\sum_{y\leq M}\frac{\phi(q)}{q}ky^{k-1}\Lambda(qy+1)e(y^k\alpha)$ and $M=\floor{N'^{1/k}}$. It can be seen from the definition of function $f$ that this expression can be split into the summation of  two integrals
\[
\int_\T\hat{\Lambda_{a,q^k}\cdot 1_{\A_{a,q^k}}}(-\alpha)\hat{\Lambda_{a,q^k}\cdot 1_{\A_{a,q^k}}}(\alpha)S_q(\alpha)\rd\alpha
\]
and
\[
-\delta\int_\T\hat \Lambda_{a,q^k}(-\alpha)\hat{\Lambda_{a,q^k}\cdot 1_{\A_{a,q^k}}}(\alpha)S_q(\alpha)\rd\alpha.
\]

First of all, it is obvious from the orthogonality principle, as well as the notation $\Lambda_{a,q^k}=\frac{\phi(q)}{q}\Lambda(a+q^k\cdot)$, that the first integral is 
\[
\Bigbrac{\frac{\phi(q)}{q}}^3\sum_{x\leq N'}\sum_{y\leq M}ky^{k-1}\Lambda(qy+1)\Lambda\cdot1_\A(a+q^kx)\Lambda\cdot1_\A(a+q^k(x+y^k)),
\]
which counts the weighted number of patterns $p_1,p_1+(p_2-1)^k$ in $\A\cap Q$ (indeed we can take $p_1=a+q^kx$, $p_2=qy+1$ so that $p_1+(p_2-1)^k=a+q^k(x+y^k)$), thus,  from our assumption, it is zero. Meanwhile, by introducing an auxiliary function $\delta\Lambda_{a,q^k}$, one may also find that the second integral is the sum of the following two representations
\[
-\delta^2 \int_\T\hat\Lambda_{a,q^k}(-\alpha)\hat\Lambda_{a,q^k}(\alpha)S_q(\alpha)\rd\alpha	
\]
and
\[
-\delta\int_\T\hat\Lambda_{a,q^k}(-\alpha)\bigbrac{\hat{\Lambda_{a,q^k}\cdot 1_{\A_{a,q^k}}}-\delta\hat\Lambda_{a,q^k}}(\alpha)S_q(\alpha)\rd\alpha=-\delta\int_\T\hat\Lambda_{a,q^k}(-\alpha)\hat f(\alpha)S_q(\alpha)\rd\alpha.	
\]
 Therefore, by putting the above analysis together we obtain that
\begin{multline*}
\int_\T\hat f(-\alpha)\hat{\Lambda_{a,q^k}\cdot 1_{\A_{a,q^k}}}(\alpha)S_q(\alpha)\rd\alpha+\delta\int_\T\hat\Lambda_{a,q^k}(-\alpha)\hat f(\alpha)S_q(\alpha)\rd\alpha\\
=-\delta^2\int_\T\hat\Lambda_{a,q^k}(-\alpha)\hat\Lambda_{a,q^k}(\alpha)S_q(\alpha)\rd\alpha.
\end{multline*}
Taking absolute values both sides and noting  $0<\delta<1$, one thus has,
\begin{multline}\label{counting-in-primes}
\delta^2\Bigabs{\int_\T\hat\Lambda_{a,q^k}(-\alpha)\hat\Lambda_{a,q^k}(\alpha)S_q(\alpha)\rd\alpha}\\
\leq\int_\T\bigabs{\hat f(\alpha)\hat\Lambda_{a,q^k}(\alpha)S_q(\alpha)}\rd\alpha+\int_\T\bigabs{\hat f(\alpha)\hat{\Lambda_{a,q^k}\cdot 1_{\A_{a,q^k}}}(\alpha)S_q(\alpha)}\rd\alpha.
\end{multline}

By truncating the integral region on the left-hand side into major arcs $\cup_{r\leq Q}\cup_{b(r)\,(b,r)=1}\set{\frac{b}{r}+\beta:|\beta|\leq\frac{Q}{rN'}}$  which are defined in (\ref{major-arc}), we can infer from asymptotics (\ref{asymp-1}) and (\ref{asymp-2}) along with Lemma \ref{majorla}, and by noting that the first term is dominant in (\ref{asymp-1}),  that when $q\ll(\log N)^A$ this integral can be asymptotically approximated by
\begin{multline*}
\sum_{r\leq Q}\Bigbrac{\frac{\phi(q)}{\phi(qr)}}^3\sum_{1\leq b\leq r\atop(b,r)=1}\biggabs{\sum_{s(r)\atop(sq+a,r)=1}e\bigbrac{\frac{sb}{r}}}^2\biggbrac{\sum_{s(r)\atop(sq+1,r)=1}e\bigbrac{\frac{s^kb}{r}}}\\
\biggbrac{\int_{|\beta|\leq1/2}\Bigbrac{\int_1^{N'}e(t\beta)\rd t}^2\int_1^{N'}e(-t\beta)\rd t\rd\beta+O\Bigbrac{\int_{\frac{Q}{rN'}<|\beta|\leq1/2}|\beta|^{-3}\,\rd\beta}}.	
\end{multline*}
We calculate the error term first. The application of Lemma \ref{cq} suggests this term is bounded by
\[
\sum_{r\leq Q}\phi(r)^{-2}r^{1-\frac{1}{k}+\eps}(rN'/Q)^2\ll N'^2Q^{-\frac{1}{k}+\eps},
\]
which is negligible. Besides, it follows from the orthogonality principle that
\begin{multline*}
\int_{|\beta|\leq1/2}\Bigbrac{\int_1^{N'}e(t\beta)\rd t}^2\int_1^{N'}e(-t\beta)\rd t\,\rd\beta=\int_{[1,N]^3}\int_\T e((t_1+t_2-t_3)\beta)\,\rd\beta\,\rd\vec t\\
=\int_{[1,N]^3}1_{t_1+t_2=t_3}\rd\vec t\gg N^2.
\end{multline*}
If we claim that when $(a,q)=1$ (we'll prove this claim in Appendix C)
\begin{align}\label{singular-series}
\sum_{r\leq Q}\biggbrac{\frac{\phi(q)}{\phi(qr)}}^3\sum_{1\leq b\leq r\atop(b,r)=1}\biggabs{\sum_{s(r)\atop(sq+a,r)=1}e\bigbrac{\frac{sb}{r}}}^2\Bigbrac{\sum_{s(r)\atop(sq+1,r)=1}e\bigbrac{\frac{s^kb}{r}}}\gg1,	
\end{align}
in  consequence of (\ref{counting-in-primes}),  at least one of the following two inequalities holds
\[
\delta^2N'^2\ll\int_\T\bigabs{\hat f(\alpha)\hat\Lambda_{a,q^k}(\alpha)S_q(\alpha)}\rd\alpha,
\]
or
\[
\delta^2N'^2\ll\int_\T\bigabs{\hat f(\alpha)\hat{\Lambda_{a,q^k}\cdot 1_{\A_{a,q^k}}}(\alpha)S_q(\alpha)}\rd\alpha.
\]

Assume that $\delta^{-2k^4-10k^3}q^k\ll\log ^AN$ in the following, the application of Proposition \ref{local inverse} with $d=q$,  $f=\Lambda_{a,q^k}\cdot1_{\A_{a,q^k}}-\delta\Lambda_{a,q^k}$,  and $\nu=\Lambda_{a,q^k}\cdot1_{\A_{a,q^k}}$ or $\nu=\Lambda_{a,q^k}$ suggests that  the second conclusion of Proposition \ref{local inverse} is true. Thus, we can assume that there is an arithmetic progression $P'=a'+q'^k\cdot[X]$ inside $[X]$ with
\[
c\delta^{4k^3+20k^2+3}\sum_{n\in P'}\Lambda(n)\leq\sum_{x\in P'}f(x),
\]
where  $(a',q')=1$, $q'\ll\delta^{-2k^3-10k^2}$ and $X\gg q'^{1-k}(\log N')^{-2^{2k}D}N'$. Setting a progression as $Q'=a+a'q^k+(qq')^k\cdot[X]$, the above inequality is  indeed
\[
c\delta^{4k^3+20k^2+3}\frac{q}{\phi(q)}\frac{q'X}{\phi(q')}\leq\sum_{n\in Q'}\Lambda\cdot1_\A(n)-\delta\sum_{n\in Q'}\Lambda(n).
\]
On noting that $X\gg N^{7/12+\eps}$ whenever $\delta^{-2k^3-10k^2}q^k\ll(\log N)^A$ and $\frac{qq'}{\phi(q')\phi(q)}\gg\frac{qq'}{\phi(qq')}$, it can be deduced from Lemma \ref{short-ap}  that
\[
\sum_{n\in Q'}\Lambda\cdot1_{\A}(n)\geq\delta\Bigbrac{1+c\delta^{4k^3+22k^2}}\sum_{n\in Q'}.\Lambda(n).
\]
\end{proof}

We are now prepared to accomplish the proof of Theorem  \ref{main}.

\vspace{2mm}

\noindent\emph{Proof of Theorem \ref{main}.}

\vspace{2mm}

Suppose that $\A\subseteq \P_N$ has relative density $\delta=\frac{|\A|}{|\P_N|}$ and $\A$ lacks patterns $p_1,p_1+(p_2-1)^k$, where $p_1,p_2,p_1+(p_2-1)^k$ are prime numbers. One may take $\A'=\A\cap[\frac{N}{\log^2 N},N]$ as a subset set of $\A$, then for every element $n\in \A'$,
\[
\Lambda(n)=\log n\geq\log N-2\log\log N.
\]
As a consequence,
\[
\sum_{n\in \A}\Lambda(n)\geq\sum_{n\in \A'}\Lambda(n)\geq (\log N-2\log\log N)|\A'|.
\]
As is evident from the Prime Number Theorem that $|\A|\geq\delta|\mathcal P_N|\geq \delta\Bigbrac{\frac{N}{\log N}-\frac{N}{\log^2N}}$, the cardinality of $\A'$ is at least 
\[
|\A'|\geq \delta \bigbrac{\frac{1}{\log N}-\frac{1}{\log^2 N}}N.
\]
Therefore, by combining the aforementioned two expressions, we obtain
\[
\sum_{n\in [N]}\Lambda\cdot1_\A(n)=\delta'N,
\]
with
\begin{align}\label{delta'}
\delta'\geq \delta\bigbrac{\frac{1}{\log N}-\frac{1}{\log^2 N}}\bigbrac{\log N-2\log\log N}\geq \delta-O\bigbrac{\frac{\delta\log\log N+1}{\log N}}.	
\end{align}

The first application of Lemma \ref{non-pseudorandomness} leads us to that either there is an arithmetic progression $Q_1=a_1+q_1^k\cdot[X_1]\subseteq[N]$ with $(a_1,q_1)=1$, $q_1\ll\delta'^{-2k^3-10k^2}$ and $X_1\gg q_1^{1-k}(\log N)^{-2^{2k}D}N$ such that
\[
\sum_{n\in Q_1}\Lambda\cdot1_\A(n)\geq\delta'\Bigbrac{1+c\delta'^{4k^3+22k^2}}\sum_{n\in Q_1}\Lambda(n).
\]

In practice, the repeated application of Lemma \ref{non-pseudorandomness}  yields a sequence of progressions: 
\[
P_i=a_1+a_2q_1^k+\cdots+(q_1\cdots q_i)^k\cdot[X_i]
\]
such that the following events happen simultaneously
\begin{enumerate}
\item 
\[
\sum_{n\in P_i}\Lambda\cdot1_\A(n)=\delta_i\sum_{n\in P_i}\Lambda(n),
\] which satisfies $\delta_i\geq\delta'\Bigbrac{1+c\delta'^{4k^3+22k^2}}^i$;
\item $a_i$ is coprime with $q_i$ and $q_i\ll\delta_{i-1}^{-2k^3-10k^2}$ and $X_i\gg(\log N)^{-2^{2k}D}q_i^{1-k}X_{i-1}$;
\item the set $\A\cap P_i$ lacks configurations $p_1,p_1+(p_2-1)^k$;
\item $(q_1\cdots q_{i})^k\ll(\log N)^A$.
\end{enumerate}

However, this iteration cannot last too long---it must stop at some step $m\ll\delta'^{-4k^3-22k^2}$, since otherwise, the relative density of $\A$ in progression $P_m$ would surpass $1$. And we'll end up (the iteration) with condition (4) fails. Because of (2) and (4), we have
\[
\bigbrac{\delta'^{-2k^3-10k^2}}^{O(\delta'^{-4k^3-22k^2})}\gg(\log N)^A,
\]
 Taking logarithms of both sides yields that
\[
\delta'\ll(\log\log N)^{-\frac{1}{4k^3+23k^2}}.
\]
The proof completes by noting that $\delta\ll(\log\log N)^{-\frac{1}{4k^3+23k^2}}$ in view of (\ref{delta'}).
\qed

\vspace{3mm}

\appendix

\section{Proof of Lemma \ref{exceptional}}

We firstly record the following three theorems,  all of them are in \cite[Chapter 4]{RS}, also in the book \cite{Dav}.

\begin{theorem}[{\cite[Proposition 4.4]{RS}}]\label{rho}
Suppose that $D\geq2$, $\chi_D$ is a primitive character of modulus $q_0\leq D$,  $\rho\geq1-\frac{c}{\log 3D}$ is a real such that  $L(\rho,\chi_D)=0$, then we call $\chi_D$ the \emph{exceptional Dirichlet character at level $D$}, and we have
\[
(1-\rho)^{-1}\ll q_0^{1/2}(\log q_0)^2.
\]	
\end{theorem}

\begin{theorem}[{\cite[Theorem 4.5]{RS}}]\label{non-principal}
There is an absolute constant $c_1>0$ such that if $D\geq2$ and $\chi$ is a non-principal character of modulus $q\leq D$, then
\begin{enumerate}
\item if the exceptional Dirichlet character $\chi_D$ exists, $L(\rho,\chi_D)=0$, and $\chi$ is induced by $\chi_D$, then for any real $x\geq1$ we have
\[
 \psi(x,\chi):=\sum_{n\leq x}\Lambda(n)\chi(n))=-\frac{x^{\rho}}{\rho}+O\Bigbrac{x\exp\bigbrac{-\frac{c_1\log x}{\sqrt{\log x}+\log D}}(\log D)^2},
\]
where $\rho$ is the exceptional zero;
\item if the exceptional Dirichlet character $\chi_D$ does not exist, or $\chi$ is not induced by $\chi_D$, then for any real $x\geq1$ we have
\[
 \psi(x,\chi)\ll x\exp\bigbrac{-\frac{c_1\log x}{\sqrt{\log x}+\log D}}(\log D)^2.
\]	
\end{enumerate}
	
\end{theorem}

\begin{theorem}[{\cite[Theorem 4.6]{RS}}]\label{principal}
There is an absolute constant $c_2>0$ such that if $\chi_0$ is the principal character of modulus $q$, then for all real $x\geq1$ we have
\[
 \psi(x,\chi_0)=x+O\Bigbrac{x\exp(-c_2\sqrt{\log x})+\log q\log x}.
\]	
\end{theorem}

\vspace{3mm}

\noindent\emph{Proof of Lemma \ref{exceptional}.}

By making use of the orthogonality principle, one has
\[
\psi(x;q,a)=\frac{1}{\phi(q)}\sum_{\chi}\bar{\chi}(a)\sum_{n\leq x}\Lambda(n)\chi(n)=\frac{1}{\phi(q)}\sum_{\chi}\bar{\chi}(a)\psi(x;\chi),
\]
where the summation is over all Dirichlet characters $\chi$ of modulus $q$. Suppose that $\chi_0\pmod q$ is the principal character, one then has,
\[
\psi(x;q,a)=\frac{\bar{\chi_0}(a)}{\phi(q)}\psi(x,\chi_0)+\frac{1}{\phi(q)}\sum_{\chi\neq\chi_0}\bar{\chi}(a)\psi(x;\chi).
\]

Let $c_0=\min\set{c_1,c_2}$, the application of Theorem \ref{principal} gives us that
\[
\frac{\bar{\chi_0}(a)}{\phi(q)}\psi(x,\chi_0)=\frac{\bar{\chi_0}(a)x}{\phi(q)}+O\Bigbrac{x\exp(-c_0\sqrt{\log x})+\log q\log x}.
\]
Now we calculate the non-principle characters. If there is a character $\chi\pmod q$ induced by the exceptional character, namely $\chi_{D_0}\pmod{q_0}$, then $q_0\mid q$, thus, $q_0\leq D_0$ and it follows from Theorem \ref{rho} that $\chi_{D_0}$ is the exceptional character of level $D_0$ and $(1-\rho)^{-1}\ll q_0^{1/2}(\log q_0)^2$. The application of Theorem \ref{non-principal} for this character $\chi=\chi_0\chi_{D_0}$ shows that
\[
 \psi(x,\chi_0\chi_{D_0})=-\frac{x^{\rho}}{\rho}+O\Bigbrac{x\exp\bigbrac{-\frac{c_0\log x}{\sqrt{\log x}+\log D_0}}(\log D_0)^2};
\]
otherwise, let $q_0=D_0+1\leq D_1$ and let $\chi_{D_0}$ be a primitive character of modulus $q_0$. it follows from the second result of Theorem \ref{non-principal} that for all non-principal characters $\chi\pmod q$
\[
\psi(x,\chi)\ll x\exp\bigbrac{-\frac{c_0\log x}{\sqrt{\log x}+\log D_0}}(\log D_0)^2.
\]
 The combination of the above four expressions gives Lemma \ref{exceptional}.

\qed

\section{Restriction for prime numbers}

The main task of this appendix is to prove  Theorem \ref{thm2} (1).  The same as Lemma \ref{spec}, we'll define the large spectrum set. For any real number $0<\eta\ll1$,  the $\eta$-large spectrum for the Fourier transform of $\nu$ is defined to be the set
\[
\R_\eta=\bigset{\alpha\in\T:|\hat\nu(\alpha)|\geq\eta X}.
\]
As expected, if we can show that
\begin{align}\label{spectrum}
\text{meas} \R_\eta\ll_{\eps}\eta^{-2-\eps}X^{-1},
\end{align}
then by making use of dyadic argument we can prove Theorem \ref{thm2} (1). Indeed, assume that $p>2$ is a real number, and let $\eps=\eps(p)$ be sufficiently small, one has, for some absolute constant $C>0$,
\[
\int_\T|\hat\nu(\alpha)|^p\rd\alpha\leq\sum_{j\geq-C}\int_{\bigset{\alpha\in\T:\frac{X}{2^{j+1}}\leq|\hat\nu(\alpha)|\leq\frac{X}{2^{j}}}}|\hat\nu(\alpha)|^p\rd\alpha\ll_{\eps,p} X^{p-1}\sum_{j\geq-C}2^{j(p-2+\eps)}\ll_{\eps,p} X^{p-1}.
\]
And this is Theorem \ref{thm2} (1).

Noting that Parseval's identity provides a crude bound, that if $\eta\ll(\log N)^{-1/\eps}$ (\ref{spectrum}) follows. Indeed, since $|\nu|\leq\Lambda_{b,d}$ pointwise, applying Parseval's identity and pigeonhole principle successively, one has
\[
\norm{\hat\nu}_2^2=\norm{\nu}_2^2=\sum_{x\in[X]}|\nu(x)|^2\leq\norm{\Lambda_{b,d}}_\infty\sum_{x\in[X]}\Lambda_{b,d}(x)\ll X\log X.
\]
For another thing, it can be seen from the definition of $\eta$-large spectrum set that
\[
(\eta X)^2\text{meas}\R_\eta\leq\int_\R|\hat\nu(\alpha)|^2\rd\alpha\leq\int_\T|\hat\nu(\alpha)|^2\rd\alpha.
\]
On combining the above two inequalities to get that
\[
\text{meas}\R_\eta\leq\eta^{-2}(\log X)X^{-1}.
\]
Thus, we may assume that
\begin{align}\label{lowerb}
(\log X)^{-1/\eps}\ll\eta\ll1.
 \end{align}
Let $\alpha_1,\cdots,\alpha_R$ be $X^{-1}$-spaced points in $\R_\eta$, in order to prove (\ref{spectrum}) it suffices to show that
\begin{align}\label{r}
	R\ll_{\eps}\eta^{-2-\eps}.
\end{align}

In the next step, we would transfer the Fourier analysis from the uncertain function $\nu$ to $\Lambda_{b,d}$. For this reason, we are now going to partition $\T$ into major and minor arcs, and show that the Fourier transform of the majorant function $\hat\Lambda_{b,d}$  possesses logarithmic savings on the minor arcs. Let $\alpha\in\T$ be a frequency, we see from Fourier transform and the definition of $\Lambda_{b,d}$ that
\[
\hat\Lambda_{b,d}(\alpha)=\sum_{x\in[X]}\Lambda_{b,d}(x)e(\alpha x)=\sum_{x\in[X]}\frac{\phi(d)}{d}\Lambda(b+dx)e(\alpha x).
\]
Let 
\begin{align*}
\majors&=\bigcup_{q\leq Q}\bigset{\alpha\in\T:\norm{q\alpha}\ll Q/X},
\end{align*}
and the minor arcs $\minors$ are the complement of major arcs, i.e. $\minors=\T\backslash\majors$, where $Q=(\log X)^{5C}$ for some large number $C>0$. 

\begin{lemma}[Major-arc asymptotic for $\hat\Lambda_{b,d}$]\label{majorla}
Suppose that $\alpha=\frac{a}{q}+\beta\in\majors$   and $1\leq b\leq d\leq\exp\bigbrac{c'\sqrt{\log X}}$ for some $0<c'<1$, then
\[
|\hat\Lambda_{b,d}(\alpha)|\ll q^{-1+\eps}X(1+X|\beta|)^{-1}+X\exp\bigbrac{-c\sqrt{\log X}}(1+X|\beta|).
\]

\end{lemma}

\begin{proof}
Due to the proof of this lemma is conceptually the same as Lemma \ref{major lemma}, we'd like to omit some details. If $\beta=0$ and $\alpha=\frac{a}{q}$,
\[
\hat\Lambda_{b,d}(a/q)=\frac{\phi(d)}{d}\sum_{r(q)\atop(b+rd,q)=1}e\bigbrac{ar/q}\Bigset{\psi(b+dX;dq,b+rd)-\psi(b+d;dq,b+rd)}.
\]

On noticing that the discrete interval $[b+d]$ contains at most one element which is congruent to $b+rd$ modulus $dq$, thereby,
\[
\hat\Lambda_{b,d}(a/q)=\frac{\phi(d)}{d}\sum_{r(q)\atop(b+rd,q)=1}e\bigbrac{ar/q}\psi(b+dX;dq,b+rd)+O(q).
\] 
Setting $c'=c_0/100$, $D_0=\exp\bigbrac{c'\sqrt{\log X}}$ and $D_1=\exp\bigbrac{2c'\sqrt{\log X}}$ in Lemma \ref{exceptional}, it then follows from summation by parts and Lemma \ref{exceptional} that for $\alpha=\frac{a}{q}+\beta$, when $q_0|q$,
\begin{multline*}
\hat\Lambda_{b,d}(\alpha)=\frac{\phi(d)}{\phi(dq)}\sum_{r(q)\atop(b+rd,q)=1}e\bigbrac{ar/q}\biggset{\int_1^Xe(\beta t)\rd t-\bar\chi(b+dr)\int_1^X(b+dt)^{\rho-1}e(\beta t)\rd t}\\
+O\Bigbrac{X\exp\bigbrac{-c\sqrt{\log X}}(1+X|\beta|)};
\end{multline*}
and when $q_0\nmid q$,
\[
\hat\Lambda_{b,d}(\alpha)=\frac{\phi(d)}{\phi(dq)}\sum_{r(q)\atop(b+rd,q)=1}e\bigbrac{ar/q}\int_1^Xe(\beta t)\rd t+O\Bigbrac{X\exp\bigbrac{-c\sqrt{\log X}}(1+X|\beta|)}.
\]
We'll finish the proof by making use of Lemma \ref{exceptional} and Lemma \ref{cq}, which is in a similar path of Lemma \ref{major lemma}.

\end{proof}

The minor arcs analysis is conceptually the same as Lemma \ref{minorsd}.  Again, we outline the main steps.

\begin{lemma}[Minor arcs estimates]\label{minor}
\[
\sup_{\alpha\in\minors}\bigabs{\hat\Lambda_{b,d}(\alpha)}\ll X(\log X)^{-C}.
\]	

\begin{proof}
It is immediate from Fourier transform as well as the definition of $\Lambda_{b,d}$ that
\[
\hat\Lambda_{b,d}(\alpha)=\frac{\phi(d)}{d}\sum_{x\in[X]}\Lambda(b+dx)e(x\alpha).
\]
After changing variables $b+dx\mapsto x$, we have
\[
\hat\Lambda_{b,d}(\alpha)=\frac{\phi(d)}{d}e\Bigbrac{-\frac{b\alpha}{d}}\sum_{x\in[b,b+dX]\atop x\equiv b\pmod d}\Lambda(x)e\Bigbrac{\frac{\alpha x}{d}}.
\]
On noting that $b$ is quite small compared with $X$, to prove this lemma it suffices to show that
\begin{align}\label{bd}
\biggabs{\sum_{x\in[b+dX]\atop x\equiv b\pmod d}\Lambda(x)e\Bigbrac{\frac{\alpha x}{d}}}\ll X(\log X)^{-C}	
\end{align}
for all frequencies $\alpha\in\minors$. And the proof is a modification of Lemma \ref{minorsd} and \cite[Lemma 4.9]{Gr05a}. To begin with, it follows from \cite[formula (3)]{Gr05b} that
\begin{multline*}
\sum_{x\in[b+dX]\atop x\equiv b\pmod d}\Lambda(x)e\Bigbrac{\frac{\alpha x}{d}}=\sum_{x\leq U\atop x\equiv b\pmod d}\Lambda(x)e\Bigbrac{\frac{\alpha x}{d}}+\threesum{xy\in[b+dX]}{x\leq U}{xy\equiv b\pmod d}\mu(x)(\log y)e\Bigbrac{\frac{\alpha xy}{d}}\\
+\threesum{xy\in[b+dX]}{x\leq U^2}{xy\equiv b\pmod d}f(x)e\Bigbrac{\frac{\alpha xy}{d}}+\threesum{xy\in[b+dX]}{x,y> U}{xy\equiv b\pmod d}\mu(x)g(y)e\Bigbrac{\frac{\alpha xy}{d}},
\end{multline*}
where $f,g$ are arithmetic functions defined in \cite[formula (4-5)]{Gr05b} satisfying $\norm{f}_\infty,\norm{g}_\infty\ll\log X$. Practically, the above four terms can be re-grouped and departed into the following type \uppercase\expandafter{\romannumeral1} and type \uppercase\expandafter{\romannumeral2} sums,
\[
\threesum{xy\in[b+dX]}{x\sim L,y\geq T}{xy\equiv b\pmod d}f(x)e\Bigbrac{\frac{\alpha xy}{d}}
\]  
and
\[
\threesum{xy\in[b+dX]}{x\sim L}{xy\equiv b\pmod d}f(x)g(y)e\Bigbrac{\frac{\alpha xy}{d}}.
\]

Next, we turn our attention to the behavior of frequencies $\alpha$ in minor arcs temporarily. Given Dirichlet's approximate theorem, for any frequency $\alpha\in\T$, there always exists an integer $q\leq X/Q$ such that $\norm{q\alpha}\ll Q/X$. Thus, when $\alpha$ belongs to minor arcs $\minors$, we must have
\begin{align}\label{qq}
Q\leq q\leq X/Q.
\end{align}
The proof of \cite[Lemma 4]{Gr05b} tells us that (\ref{bd}) follows from the following two claims and (\ref{qq}).
 
\subsection*{Claim 1} Suppose that $\bigabs{\alpha-\frac{a}{q}}\leq\frac{1}{q^2}$ with $(a,q)=1$ and $1\leq a\leq q$. Suppose also that $X\geq L$, then we have
\[
\threesum{xy\in[b+dX]}{x\sim L,y\geq T}{xy\equiv b\pmod d}f(x)e\Bigbrac{\frac{\alpha xy}{d}}\ll(\log X)(\log q)\norm{f}_\infty\bigbrac{\frac{X}{q}+L+q}.
\]
\subsection*{Claim 2} Suppose that $\bigabs{\alpha-\frac{a}{q}}\leq\frac{1}{q^2}$ with $(a,q)=1$ and $1\leq a\leq q$. Suppose also that $X\geq L$, then we have
\[
\threesum{xy\in[b+dX]}{x\sim L}{xy\equiv b\pmod d}f(x)g(y)e\Bigbrac{\frac{\alpha xy}{d}}\ll\norm{f_\infty}\norm{g}_\infty X(\log q)\Bigbrac{\frac{1}{q}+\frac{L}{X}+\frac{q}{X}}^{1/4}.
\]
To prove Claim 1, we  notice that
\begin{multline*}
\Biggabs{\threesum{xy\in[b+dX]}{x\sim L,y\geq T}{xy\equiv b\pmod d}f(x)e\Bigbrac{\frac{\alpha xy}{d}}}\leq \sum_{x\sim L\atop(x,d)=1}f(x)\biggabs{\sum_{T\leq y\leq\frac{b+dX}{x}\atop y\equiv bx^{-1}\pmod d}e\Bigbrac{\frac{\alpha xy}{d}}}\\
\ll\norm{f}_\infty\sum_{x\sim L}\Bigabs{\sum_{\frac{T}{d}\leq y\leq \frac{X}{x}}e(xy\alpha)}\ll\norm{f}_\infty\sum_{x\sim L}\min\bigset{\frac{X}{x},\norm{x\alpha}^{-1}}.
\end{multline*}
Thus, Claim 1 follows from \cite[Lemma 2]{Gr05b}. It remains to prove Claim 2. One may obtain from an application of Cauchy-Schwarz inequality that
\[
\Biggabs{\threesum{xy\in[b+dX]}{x\sim L}{xy\equiv b\pmod d}f(x)g(y)e\Bigbrac{\frac{\alpha xy}{d}}}^2\leq\sum_{x\sim L}f(x)^2\sum_{s\sim L\atop(x,d)=1}\sum_{y_1,y_2\leq\frac{b+dX}{x}\atop y_1,y_2\equiv bx^{-1}\pmod d}g(y_1)g(y_2)e\Bigbrac{\frac{\alpha x(y_1-y_2)}{d}}.
\]
Changing variables and exchanging the order of summation yields that above expression is bounded by
\[
L\norm{f}_\infty^2\norm{g}_\infty^2\sum_{y_1,y_2\leq\frac{X}{L}}\Bigabs{\sum_{x\sim L\atop (x,d)=1}e(x(y_1-y_2)\alpha)}.
\]
After square (both sides) and another application of Cauchy-Schwarz inequality gives
\begin{multline*}
\Biggabs{\threesum{xy\in[b+dX]}{x\sim L}{xy\equiv b\pmod d}f(x)g(y)e\Bigbrac{\frac{\alpha xy}{d}}}^4\ll\norm{f}_\infty^4\norm{g}_\infty^4X^2\sum_{y_1,y_2\leq\frac{X}{L}}\sum_{x\sim L\atop(x,d)=1}\sum_{t\leq 2L}e(t(y_1-y_2)\alpha)\\
\ll	\norm{f}_\infty^4\norm{g}_\infty^4X^3\sum_{j\leq\frac{X}{L}}\min\bigset{L,\norm{j\alpha}^{-1}}.
\end{multline*}
One then deduces from \cite[Lemma 1]{Gr05b} that
\[
\Biggabs{\threesum{xy\in[b+dX]}{x\sim L}{xy\equiv b\pmod d}f(x)g(y)e\Bigbrac{\frac{\alpha xy}{d}}}^4\ll\norm{f}_\infty^4\norm{g}_\infty^4X^4(\log q)\Bigbrac{\frac{1}{q}+\frac{L}{X}+\frac{q}{X}}.
\]
\end{proof}

\end{lemma}

\vspace{2mm}

We are now left to verify  (\ref{r}). For $1\leq r\leq R$, let $c_r\in\set{z\in\C:|z|=1}$ be a number such that $c_r\hat\nu(\alpha_r)=|\hat\nu(\alpha_r)|$. And let $l_x(1\leq x\leq X)\in\C$ be a sequence such that $|l_x|\leq1$ and $l_x\Lambda_{b,d}(x)=\nu(x)$. Routinely, as in Section 5, one has
\begin{multline*}
R^2\eta^2X^2	\leq\biggbrac{\sum_{r\in[R]}|\hat\nu(\alpha_r)|}^2\ll\biggabs{\sum_{r\in[R]}c_r\sum_{x\in[X]}l_x\Lambda_{b,d}(x)e(x\alpha_r)}^2\\
\ll X\sum_{x\in[X]}\biggabs{\sum_{r\in[R]}c_r\Lambda_{b,d}(x)e(x\alpha_r)}^2\ll X\sum_{r,r'\in[R]}\Bigabs{\sum_{x\in[X]}\Lambda_{b,d}(x)e\biggbrac{(\alpha_r-\alpha_{r'})x}},
\end{multline*}
just on noticing that $\sum_{x\in[X]}l_x^2\Lambda_{b,d}(x)\ll X$.
It then follows from Fourier transform that 
\[
R^2\eta^2X\ll\sum_{r,r'\in[R]}\bigabs{\hat\Lambda_{b,d}(\alpha_r-\alpha_{r'})}.
\]
Put $\gamma=1+\eps$, utilizing H\"older's inequality one  then has
\[
R^{2}\eta^{2\gamma}X^\gamma\ll\sum_{r,r'\in[R]}\bigabs{\hat\Lambda_{b,d}(\alpha_r-\alpha_{r'})}^\gamma.
\] 
Joint Lemma \ref{majorla}, Lemma \ref{minor}, (\ref{lowerb}) and putting $\widetilde Q=\eta^{-1-\eps}$ gives
\[
\eta^{2\gamma}R^2\ll\sum_{q\leq \widetilde Q}\sum_{a(q)\atop(a,q)=1}\sum_{1\leq r,r'\leq R}q^{-\gamma+\eps}(1+X\norm{\alpha_r-\alpha_{r'}-a/q})^{-\gamma}\ll\sum_{1\leq r,r'\leq R}G(\alpha_r-\alpha_{r'}),
\]
where 
\[
G(\alpha)=\sum_{q\leq \widetilde Q}\sum_{a(q)}q^{-\gamma+\eps}F(\alpha-a/q)
\]
and
\[
F(\beta)=(1+X|\sin\beta|)^{-\gamma}.
\]
We are now in the position of \cite[Eq.(4.16)]{Bou89}, just need to replace $N^2$ by $X$, replace $\delta$ by $\eta$. Following from Bourgain's proof, on noting $\gamma=1+\eps$, one then has $R\ll_\eps \eta^{-2-\eps}$.

\section{Estimation of the singular series (\ref{singular-series})}

Since the idea is rather similar to the book \cite[Section 8]{Hua}, the proof here would be a bit rough. On writing
\[
A(r)=\biggbrac{\frac{\phi(q)}{\phi(qr)}}^3\sum_{1\leq b\leq r\atop(b,r)=1}\biggabs{\sum_{s(r)\atop(sq+a,r)=1}e\bigbrac{\frac{sb}{r}}}^2\Bigbrac{\sum_{s(r)\atop(sq+1,r)=1}e\bigbrac{\frac{s^kb}{r}}},
\]
it is multiplicative. Suppose that $(a,q)=1$ and $Q=(\log N)^A$ for some large $A>1$,  the goal is to show that  
\[
\sum_{r\leq Q}A(r)\gg1.
\]

Assume for a moment that $r=p^t$, then if $p|q$, $(sq+a,p)=(a,p)=1$ and $(sq+1,p)=1$ holds for all choices of  $s\pmod {p^t}$ whenever $(a,q)=1$, and thus, in this case,
\[
A(p^t)=\sum_{1\leq b\leq p^t\atop(b,p)=1}\biggabs{\sum_{s(p^t)}e\bigbrac{\frac{sb}{p^t}}}^2\Bigbrac{\sum_{s(p^t)}e\bigbrac{\frac{s^kb}{p^t}}}=0.
\]
If $p\nmid q$ and degree $t\geq 2$, we may write $s\pmod {p^t}$ as the sum $s=s_1p^{t-1}+s_2$, and when $(b,p)=1$ we have
\[
\sum_{s(p^t)\atop(sq+a,p)=1}e\bigbrac{\frac{sb}{p^t}}=\sum_{s_1(p)}\sum_{s_2(p^{t-1})\atop(s_2q+a,p)=1}e\Bigbrac{\frac{s_1b}{p}+\frac{s_2b}{p^{t}}}=0,
\]
which means that $A(p^t)\equiv0$ whenever degree $t\geq 2$. 
Putting the above analysis together one thus has
\begin{align}\label{r-up-q}
\sum_{r\leq Q}A(r)=\prod_{p\leq Q\atop p\nmid q}\bigbrac{1+A(p)}+O\biggbrac{\threesum{r>Q}{r:\Box-\mathrm{free}}{p|r\Rightarrow p\leq Q,p\nmid q}|A(r)|},
\end{align}
where 
\[
A(p)=\phi(p)^{-3}\sum_{1\leq b\leq p\atop(b,p)=1}\biggabs{\sum_{s(p)\atop(sq+a,p)=1}e\bigbrac{\frac{sb}{p}}}^2\Bigbrac{\sum_{s(p)\atop(sq+1,p)=1}e\bigbrac{\frac{s^kb}{p}}}\qquad (p\nmid q).
\] 
We shall deal with the big O-term and the ``main term"  successively, and in the following, we always suppose that primes $p$ satisfy $p\nmid q$. Firstly, a modification of the proof of \cite[Lemma 8.5]{Hua} shows that when $(b,p)=1$,
\[
\biggabs{\sum_{s(p)\atop(sq+1,p)=1}e\bigbrac{\frac{s^kb}{p}}}\ll p^{1/2+\eps};
\]
besides, under the same condition, we also have
\[
\biggabs{\sum_{s(p)\atop(sq+a,p)=1}e\bigbrac{\frac{sb}{p}}}\leq\biggabs{\sum_{s(p)\atop(s,p)=1}e\bigbrac{\frac{sb}{p}}}+\biggabs{\sum_{s(p)\atop p|(sq+a)}e\bigbrac{\frac{sb}{p}}}\leq2.
\]
As a consequence,
\begin{align}\label{ap}
|A(p)|\ll p^{-3/2+\eps}.
\end{align}
Therefore, the big O-term in (\ref{r-up-q}) can be bounded up by $o(1)$, in practice since $A(r)$ is multiplicative, 
\[
\threesum{r>Q}{r:\Box-\mathrm{free}}{p|r\Rightarrow p\leq Q,p\nmid q}|A(r)|\ll\sum_{r>Q\atop r:\Box-\mathrm{free}}\prod_{p|r\atop p\nmid q}|A(p)|\ll\sum_{r>Q}r^{-3/2+\eps}\ll Q^{-1/2+\eps}.
\]
By substituting the above inequality into (\ref{r-up-q}), we may find that  it is sufficient  to show that for some large constant $C>1$
\begin{align}\label{p-c-q}
\prod_{p\leq Q\atop p\nmid q}\bigbrac{1+A(p)}=\prod_{p\leq C\atop p\nmid q}\bigbrac{1+A(p)}\prod_{C\leq p\leq Q\atop p\nmid q}\bigbrac{1+A(p)}\gg 1.
\end{align}

In order to deal with this, let $M(p)$ count the number of solutions of the following equation
\[
x_1-x_2+x_3^k\equiv 0\pmod p
\]
with variables $x_1,x_2,x_3\pmod p$ and $p\nmid(qx_1+a)(qx_2+a)(qx_3+1)$. Using orthogonality, one has 
\begin{align*}
pM(p)&=\sum_{h(p)}\biggabs{\sum_{x(p)\atop(xq+a,p)=1}e\bigbrac{\frac{xh}{p}}}^2\Bigbrac{\sum_{x(p)\atop(xq+1,p)=1}e\bigbrac{\frac{x^kh}{p}}}\\
&=\Bigset{\sum_{h(p)\atop(h,p)=1}+\sum_{h(p)\atop(h,p)=p}}\biggabs{\sum_{x(p)\atop(xq+a,p)=1}e\bigbrac{\frac{xh}{p}}}^2\Bigbrac{\sum_{x(p)\atop(xq+1,p)=1}e\bigbrac{\frac{x^kh}{p}}}\\
&=\phi(p)^3(A(p)+1),
\end{align*}
just on recalling the definition of $A(p)$. Thus, to prove (\ref{p-c-q}) it suffices to establish that
\[
\prod_{p\leq C\atop p\nmid q}\frac{p}{\phi(p)^3}M(p)\prod_{C\leq p\leq Q\atop p\nmid q}\bigbrac{1+A(p)}\gg 1.	
\]
The next step is to show that $M(p)\geq 1$  for primes $p\nmid q$. In practice, when $p\geq 3$, on noting that $\set{x\pmod p,p\nmid(qx+a)}$ runs over $\phi(p)$ residue classes modulus $p$, the using of \cite[Lemma 8.7]{Hua} yields that the sequence $\set{x_1+x_2\pmod p,p\nmid(qx_1+a)(qx_2+a)}$ runs over  at least $\min\set{2\phi(p)-1,p}\geq p$ residue classes modulus $p$. Therefore, there must be at least one choice of $x_1,x_2,x_3\pmod p$ such that
\[
\begin{cases}
x_1-x_2+x_3^k\equiv 0\pmod p	\\
p\nmid(qx_1+a)(qx_2+a)(qx_3+1).
\end{cases}
\]
When $p=2$, it is clearly that $1,1,0\pmod 2$ is a solution to the system
\[
\begin{cases}
x_1-x_2+x_3^k\equiv 0\pmod 2\\
2\nmid(qx_1+a)(qx_2+a)(qx_3+1).
\end{cases}
\]
The conclusion $M(p)\geq1$ when $p\nmid q$ then yields that
 \[
\prod_{p\leq C\atop p\nmid q}\frac{p}{\phi(p)^3}M(p)\prod_{C\leq p\leq Q\atop p\nmid q}\bigbrac{1+A(p)}\gg\prod_{p\leq C\atop p\nmid q} p^{-2+\eps}\prod_{C\leq p\leq Q\atop p\nmid q}\bigbrac{1+A(p)}.  
\]
Taking note from (\ref{ap}) that $1+A(p)\geq 1-Cp^{-3/2+\eps}$ and that $C$ is a constant, thus,
\[
\prod_{p\leq C\atop p\nmid q} p^{-2+\eps}\prod_{C\leq p\leq Q\atop p\nmid q}\bigbrac{1+A(p)}\gg\prod_{ p\geq C}\bigbrac{1-O\bigbrac{p^{-3/2+\eps}}}\gg1.
\]

\section{Proof of Lemma \ref{minorsd}}

For one thing, given Dirichlet approximate theorem, for any frequency $\alpha\in\T$, there always exists an integer $q\leq N'/Q$ such that $\norm{q\alpha}\ll Q/N'$. Thus, when $\alpha$ belongs to minor arcs $\minor$, we must have
\begin{align}\label{q}
Q\leq q\leq N'/Q.
\end{align}

For another, from the proof of Lemma \ref{major lemma}, it is clear that
\[
S_d(\alpha)=\sum_{y\in[dM+1]\atop y\equiv 1\pmod d}k\frac{\phi(d)}{d}\Bigbrac{\frac{y-1}{d}}^{k-1}\Lambda(y)e\Bigbrac{\alpha\Bigbrac{\frac{y-1}{d}}^k}.
\]
After making use of Abel's summation formula, we'll get
\begin{multline*}
S_d(\alpha)=	\frac{\phi(d)}{d}kM^{k-1}\sum_{y\in[dM+1]\atop y\equiv 1\pmod d}\Lambda(y)e\Bigbrac{\alpha\Bigbrac{\frac{y-1}{d}}^k}\\
-\int_1^{dM+1}\biggbrac{\sum_{y\in[t]\atop y\equiv 1\pmod d}\Lambda(y)e\Bigbrac{\alpha\Bigbrac{\frac{y-1}{d}}^k}}\cdot\biggbrac{\frac{\phi(d)}{d}k\Bigbrac{\frac{t-1}{d}}^{k-1}}^\prime\rd t.
\end{multline*}
From the perspecitive of  Brun\,-Titchmarsh theorem (\cite[Theorem 2]{MV}) it is an easy matter to verify that
\[
\biggabs{\sum_{y\in[t]\atop y\equiv 1\pmod d}\Lambda(y)e\Bigbrac{\alpha\Bigbrac{\frac{y-1}{d}}^k}}\leq\pi(t;d,1)\log t \ll\frac{t\log t}{\phi(d)},
\]
whilst derivation gives the upper bound
\[
\biggbrac{\frac{\phi(d)}{d}k\Bigbrac{\frac{t-1}{d}}^{k-1}}^\prime\ll\frac{\phi(d)}{d^k}t^{k-2}.
\]
Combining the above two inequalities it is not hard to find that the integral over the interval $[1,dM(\log N')^{-2D}]$ is negligible, in practice,
\begin{multline*}
\left|\int_1^{dM(\log N')^{-2D}}\biggbrac{\sum_{y\in[t]\atop y\equiv 1\pmod d}\Lambda(y)e\Bigbrac{\alpha\Bigbrac{\frac{y-1}{d}}^k}}\cdot\biggbrac{\frac{\phi(d)}{d}k\Bigbrac{\frac{t-1}{d}}^{k-1}}^\prime\rd t\right|\\
\ll d^{-k}\int_1^{dM(\log N')^{-2D}}t^{k-1}\log t\rd t\ll N'(\log N')^{-D},
\end{multline*}
on recalling that $N'\asymp M^k$. Therefore,  $S_d(\alpha)$ is bounded by
\begin{multline}\label{11}
\frac{\phi(d)}{d}M^{k-1}\sup_{dM(\log N')^{-2D}\leq t \leq dM+1}\biggabs{\sum_{y\in[t]\atop y\equiv 1\pmod d}\Lambda(y)e\Bigbrac{\alpha\Bigbrac{\frac{y-1}{d}}^k}}+O\bigbrac{N'(\log N')^{-D}}\\
\leq \frac{\phi(d)}{d} M^{k-1}\biggset{\biggabs{\sum_{y\leq \frac{dM\log^{-2D}N'}{2}\atop y\equiv 1\pmod d}\Lambda(y)e\Bigbrac{\alpha\Bigbrac{\frac{y-1}{d}}^k}}\\
+\sup_{dM(\log N')^{-2D}\leq t \leq dM+1}\biggabs{\sum_{\frac{dM\log^{-2D}N'}{2}<y\leq t\atop y\equiv 1\pmod d}\Lambda(y)e\Bigbrac{\alpha\Bigbrac{\frac{y-1}{d}}^k}}} +O\bigbrac{N'(\log N')^{-D}}.
\end{multline}
It is a fairly immediate consequence of Brun\,-Titchmarsh theorem that the first term is bounded by $O\bigbrac{N'\log^{-2D}N'}$. We then  decompose the second term  into dyadic intervals to get that
\[
\biggabs{\sum_{\frac{dM\log^{-2D}N'}{2}<y\leq t\atop y\equiv 1\pmod d}\Lambda(y)e\Bigbrac{\alpha\Bigbrac{\frac{y-1}{d}}^k}}\ll\log N'\biggabs{\sum_{y\sim T\atop y\equiv1\pmod d}\Lambda(y)e\Bigbrac{\alpha\Bigbrac{\frac{y-1}{d}}^k}}
\]
with  $dM(\log N')^{-2D}/2\leq T \leq dM/2$.   It can be deduced from Vaughan's identity with parameter $U=\Bigbrac{\frac{T}{d}}^{\frac{1}{2^k}}$ that the inner sum of the above right-hand side is equal to
\begin{multline*}
\threesum{uv\sim T}{uv\equiv 1\pmod d}{u\leq U}\mu(u)(\log v)e\Bigbrac{\alpha\Bigbrac{\frac{uv-1}{d}}^k}
-\threesum{s\leq U^2}{sw\sim T}{sw\equiv 1\pmod d}a(s)e\Bigbrac{\alpha\Bigbrac{\frac{sw-1}{d}}^k}\\+\threesum{sv\sim T}{s,v>U}{sv\equiv1\pmod d}b(s)\Lambda(v) e\Bigbrac{\alpha\Bigbrac{\frac{sv-1}{d}}^k}:=I_1-I_2+I_3,	
\end{multline*}
where $a(s)=\sum_{uv=s\atop u,v\leq U}\mu(u)\Lambda(v)$ and $b(s)=\sum_{uw=s\atop u>U}\mu(u)$ are two arithmetic functions.

 We now estimate $I_1$, $I_2$ and $I_3$ in turns. By decomposing the interval $u\leq U$ in $I_1$ into dyadic intervals $u\sim U_i$ and $U_i\leq U/2$,  rescaling $U$ by $2U$ we may conclude that
\[
|I_1|\ll\log U\log T\sum_{u\sim U\atop (u,d)=1}\sup_{1\leq z\leq \frac{2T}{u}}\biggabs{\sum_{\max\set{z,\frac{T}{u}}<v\leq\frac{2T}{u}\atop v\equiv u^{-1}\mod d}e\Bigbrac{\alpha\Bigbrac{\frac{uv-1}{d}}^k}}.
\]
Let $u^{-1}$ be the inverse of $u$ modulus $d$, that is $uu^{-1}\equiv1\mod d$ (so does $s^{-1}$,  $v^{-1}$ and so on). Let $c_d$ be the integer satisfying $uu^{-1}-1=dc_d$. Taking $v=u^{-1}+dv'$ with $\max\set{\frac{z}{d},\frac{T}{ud}}<v'\leq\frac{2T}{ud}$, it is immediate that
\[
|I_1|\ll\frac{T\log U\log T}{d}\E_{u\sim U} \sup_{1\leq z\leq \frac{2T}{u}} \biggabs{\E_{\max\set{\frac{z}{d},\frac{T}{ud}}<v'\leq\frac{2T}{ud}}e(\alpha(uv'+c_d)^k)},
\]
 we introduce the symbol of averaged sum here, which is defined to be $\E_{x\in A}f(x)=\frac{1}{|A|}\sum_{x\in A}f(x)$ whenever $A$ is a finite set. H\"older's inequality then implies that
\begin{multline*}
\Bigbrac{\frac{d|I_1|}{T\log U\log T}}^{2^{k-1}}\ll\E_{u\sim U}\sup_{1\leq z\leq \frac{2T}{u}}\biggabs{ \E_{\max\set{\frac{z}{d},\frac{T}{ud}}<v'\leq\frac{2T}{ud}}e(\alpha(uv'+c_d)^k) }^{2^{k-1}}\\
\ll\E_{u\sim U}	 \sup_{1\leq z\leq \frac{2T}{u}} \E_{\max\set{\frac{z}{d},\frac{T}{ud}}<v'\leq\frac{2T}{ud}} \E_{|v_1|\cdots,|v_{k-1|\leq \frac{2T}{ud}}}\\e(\bigbrac{k!v'v_1\cdots v_{k-1}+P(v_1,\cdots,v_{k-1})}u^k\alpha+P_{k-1}(u)\alpha),
\end{multline*}
where $P_{k-1}(u)$ is a polynomial of degree at most $k-1$ with respect to the variable $u$.
Switching the order of averages, another application of H\"older's inequality yields that
\begin{multline*}
\Bigbrac{\frac{d|I_1|}{T\log U\log T}}^{2^{2k-2}}\ll\\
\E_{v'\leq\frac{2T}{Ud}}\E_{|v_1|\cdots,|v_{k-1}|\leq \frac{2T}{Ud}}\E_{|u_1|,\cdots,|u_{k-1}|\leq 2U}\Bigabs{\E_{u\sim U}e(k!^2v'v_1\cdots v_{k-1}u_1\cdots u_{k-1}u\alpha)}.
\end{multline*}
If we set $m=k!^2v'v_1\cdots v_{k-1}u_1\cdots u_{k-1}$, then $m\ll\frac{T^k}{Ud^k}$ and Cauchy-Schwarz inequality shows that the above inequality is bounded by
\begin{multline*}
\Bigbrac{\frac{d|I_1|}{T\log U\log T}}^{2^{2k-2}}\ll(\frac{d}{T})^k\sum_{m\ll\frac{T^k}{Ud^k}}\tau_{3k}(m)\Bigabs{\sum_{u\sim U}e(mu\alpha)}\\
\ll (\frac{d}{T})^k \Bigbrac{\sum_{m\ll\frac{T^k}{Ud^k}}\tau_{3k}^2(m)}^{1/2}\Bigbrac{U\sum_{m\ll\frac{T^k}{Ud^k}}\bigabs{\sum_{u\sim U}e(mu\alpha)}}^{1/2},
\end{multline*}
where $\tau_k$ is the $k$-fold divisor function.
On employing \cite[Lemma 1]{Gr05b} and  the estimate $\sum_{n\leq N}\tau_k^2(n)\ll N\log^{2k-1}N$ which is an immediate conclusion of Shiu's bound \cite[Theorem 1]{Shiu}, one has
\begin{align*}
\Bigbrac{\frac{d|I_1|}{T\log U\log T}}^{2^{2k-2}}&\ll(d/T)^k	\bigbrac{\frac{T^k}{Ud^k}\log^{6k-1}T}^{1/2}U^{1/2}\bigbrac{\frac{T^k}{qd^k}+U+(\frac{T^k}{Ud^k}+q)\log q}^{1/2}\\
&\ll\log^{3k}N'\bigbrac{\frac{1}{q}+\frac{Ud^k}{T^k}+\frac{1}{U}+\frac{qd^k}{T^k}}^{1/2}.
\end{align*}
In view of (\ref{q}) and (\ref{deq}) we can conclude that $|I_1|\ll\frac{T}{d}\log^{-D}N'$. Similarly, $|I_2|\ll\frac{T}{d}\log^{-D}N'$.

As for $I_3$, a dyadic argument gives us that the following inequality holds uniformly for all $U\leq L\leq\frac{2T}{U}$ 
\begin{multline*}
|I_3|\ll\log T\biggabs{\sum_{s\sim L\atop(s,d)=1}b(s)\sum_{\max\set{\frac{T}{s}, U}<v\leq\frac{2T}{s}\atop v\equiv s^{-1}\pmod d}\Lambda(v)e\Bigbrac{\alpha\Bigbrac{\frac{sv-1}{d}}^k}}\\
\ll\log T\biggabs{ \sum_{s\sim L\atop(s,d)=1}b(s) \sum_{\max\set{\frac{T}{sd},\frac{U}{d}}<v\leq\frac{2T}{sd}}\Lambda(s^{-1}+dw)e(\alpha(su+c_d)^k)}\\
\ll\frac{T\log T}{d}\E_{s\sim L}b(s)\Bigabs{\E_{\max\set{\frac{T}{sd},\frac{U}{d}}<v\leq\frac{2T}{sd}}\Lambda(s^{-1}+dw)e(\alpha(su+c_d)^k)}.
\end{multline*}
Noting that $|b(s)|\leq\tau_2(s)$ pointwise, it follows from Cauchy-Schwarz inequality that 
\begin{multline*}
\bigbrac{\frac{d|I_3|}{T\log T}}^2\ll\E_{s\sim L}b^2(s)\E_{s\sim L}\Bigabs{\E_{\max\set{\frac{T}{sd},\frac{U}{d}}<v\leq\frac{2T}{sd}}\Lambda(s^{-1}+dw)e(\alpha(su+c_d)^k)}^2\\
\ll \log^3L\, \E_{v\leq\frac{2T}{Ld}}\E_{|h|\leq\frac{2T}{Ld}} \Lambda(s^{-1}+dv) \Lambda(s^{-1}+d(v+h))\bigabs{\E_{s\sim L}e\bigbrac{\alpha(s(v+h)+c_d)^k-\alpha(sv+c_d)^k}}\\
\ll\log^3L\log^2T\,\E_{v\leq\frac{2T}{Ld}}\E_{|h|\leq\frac{2T}{Ld}}\bigabs{\E_{s\sim L}e\bigbrac{(kv^{k-1}h+P_{k-2}(v))s^k\alpha+P_{k-1}(s)}\alpha},
\end{multline*}
here, $P_{k-2}(v)$ is a polynomial of degree at most $k-2$ with respect to $v$ and $P_{k-1}(s)$ is a polynomial of degree at most $k-1$ with respect to $s$. As in the calculation of $I_1$, we can linearize the variables $s$ and $v$ by repeatedly applying  Cauchy-Schwarz inequality to conclude that
\[
|I_3|\ll\frac{T}{d}\log^{-D}N'.
\]
The proof completes by noting that $T$ varies in the range of $dM(\log N')^{-2D}/2\leq T \leq dM/2$.

\section{Conflict of interest statement}
On behalf of all authors, the corresponding author states that there is no conflict of interest.

\end{document}